\documentclass[a4paper,11pt]{article}
\usepackage[mathscr]{eucal}
\usepackage{amssymb}
\usepackage{latexsym}
\usepackage{amsthm}
\usepackage{amsmath}
\usepackage[dvips]{graphicx}
\usepackage{psfrag}

\newcommand{\dd}{\; \mathrm{d}}

\newcommand\ip[2]{\langle#1,#2\rangle}

\newcommand\nn{{\mathbb N}}

\newcommand\Om{\Omega}

\newcommand\opk[1]{\mathop{\mathrm{#1}}\nolimits}
\newcommand\Harm{\opk{Harm}}
\newcommand\Hom{\opk{Hom}}

\newtheorem{theorem}{Theorem}
{Proposition}
\newtheorem{lemma}
{Lemma}
{Corollary}
\theoremstyle{definition}
{Definition}

{Problem}
\newtheorem{remark}
{Remark}

\newtheorem{example}
{Example}

\title{Complex spherical codes with three inner products 
}
\author{
Hiroshi Nozaki
\& 
Sho Suda \\
\quad \\
	Department of Mathematics Education \\
	Aichi University of Education\\
	1 Hirosawa, Igaya-cho, 
	Kariya, Aichi 448-8542, 
	Japan\\ 
\quad \\
	hnozaki@auecc.aichi-edu.ac.jp 
\& suda@auecc.aichi-edu.ac.jp}
\begin{document}
\maketitle

\renewcommand{\thefootnote}{\fnsymbol{footnote}}
\footnote[0]{}


\bigskip

\begin{abstract}
Let $X$ be a finite set in a complex sphere of $d$ dimension. 
Let $D(X)$ be the set of usual inner products of two distinct vectors in $X$. 
A set $X$ is called a complex spherical $s$-code if the cardinality of $D(X)$ is $s$ and 
$D(X)$ contains an imaginary number.  We would like to classify the largest possible $s$-codes for  given dimension $d$.  In this paper, we consider the problem for the case $s=3$.   
Roy and Suda (2014) gave a certain upper bound for the cardinalities of $3$-codes. 
A $3$-code $X$ is said to be tight if $X$ attains the bound.  
We show that there exists no tight $3$-code except for dimensions $1$, $2$.  
Moreover we make an algorithm to classify the largest $3$-codes by 
considering representations of oriented graphs.  
By this algorithm, the largest $3$-codes are classified for dimensions $1$, $2$, $3$ with a current computer. 
\end{abstract}

\bigskip

\noindent
\textbf{Acknowledgments.} 
Hiroshi Nozaki is supported by JSPS KAKENHI Grant Numbers 25800011, 26400003, 16K17569, 17K0515501. 
Sho Suda is supported by JSPS KAKENHI Grant Numbers 15K21075, 26400003, 17K0515501.
The authors thank  an anonymous referee for some useful comments and suggestions.

\section{Introduction}
Let $X$ be a finite set in the $d$-dimensional complex unit
sphere $\Omega(d)$ in $\mathbb{C}^d$. The {\it angle set} $D(X)$
is defined to be 
\[
D(X)=\{\boldsymbol{x}^*\boldsymbol{y} \mid  \boldsymbol{x},\boldsymbol{y} \in X, \boldsymbol{x} \ne \boldsymbol{y} \}, 
\]
where $\boldsymbol{x}^*$ is the transpose conjugate of a column vector $\boldsymbol{x}$. 
A finite set $X$ is a {\it complex spherical $s$-code} 
if $|D(X)|=s$ and $D(X)$ contains an imaginary number. The value $s$ is called the {\it degree} of $X$.  
For $X, X' \subset \Omega(d)$, we say that $X$ is  
{\it isomorphic} to $X'$ if there exists a unitary transformation
from $X$ to $X'$. 
An $s$-code $X\subset \Omega(d)$ is {\it largest} if $X$ has 
the largest possible cardinality in all $s$-codes in $\Omega(d)$.  
One of major problems on $s$-codes 
is to classify the largest $s$-codes for given $s$ and $d$. 

For the real sphere $S^{d-1}$, a similar concept to $s$-codes is well studied \cite{DGS77}.   
A subset $X$ of  $S^{d-1}$ is an {\it $s$-distance set} if $|D(X)|=s$. 
 Delsarte, Goethals, and Seidel \cite{DGS77} gave an upper bound
\[
|X| \leq \binom{d+s-1}{s}+\binom{d+s-2}{s-1}
\]
for an $s$-distance set $X$ in $S^{d-1}$. 
An $s$-distance set $X$ is {\it tight}  if $X$ attains this bound. A tight $s$-distance set
has the structure of a $Q$-polynomial association scheme, and becomes a tight spherical $2s$-design~\cite{DGS77}.
Tight $s$-distance sets have been classified except for $s=2$ \cite{BB,BD79,BMV04,NVX}. 
The largest $1$-distance set in $S^{d-1}$ is the regular simplex. The largest  $s$-distance set in $S^1$ is the regular $(2s+1)$-gon. The largest  $2$-distance set in $S^{d-1}$ has been determined for all $d$ except for $d = (2k+1)^2-3$ with $k \in \mathbb{N}$ \cite{BY13,L97,M09,GYX}. The largest $3$-distance set in $S^{d-1}$ has been determined for $d =3,8,22$ \cite{MN11,Spre}.  The largest spherical $s$-distance set is not known for other $(s,d)$. The classification of largest spherical $s$-distance sets is still open except for $(s,d)=(1,d),(s,2),(2,d\leq 7),(2,23),(3,3)$. 

We have the following upper bound for a $2$-code $X$ in $\Omega(d)$ \cite{RSX,NSudapre2}. 
\[
|X| \leq \begin{cases}
2d+1 & \text{ if $d$ is odd}, \\
2d  & \text{ if $d$ is even}. 
\end{cases}
\]
A $2$-code $X$ is {\it tight} if $X$ attains this bound. For odd $d$ ({\it resp}.\ even $d$), the existence of a tight $2$-code in $\Omega(d)$ is equivalent to that of a doubly regular tournament ({\it resp}.\ skew Hadamard matrix) of order $d$ \cite{NSudapre2}.  
We have the following upper bound for a $3$-code $X$ in $\Omega(d)$ \cite{RSX}. 
\[
|X|\leq
\begin{cases}
4 & \text{ if } d=1, \\
d^2+2d &  \text{ if }d\geq2.
\end{cases} 
\]
 A $3$-code $X$ is {\it tight} if $X$ attains this bound. Roy and Suda \cite{RSX} proved that 
a tight $3$-code has the structure of a commutative non-symmetric association scheme.  
In this paper, we show that there exists no tight 3-code except for $d= 1,2$. 

We use complex representations of oriented graphs 
in order to classify the largest $3$-codes in $\Omega(d)$.  An oriented graph is a directed graph which has no symmetric pair of directed
edges. An oriented graph $G=(V,E)$ is {\it representable in $\Omega(d)$} 
if there exist 
a mapping $\varphi$ from $V$ to $\Omega(d)$,
an imaginary number $\alpha$ with ${\rm Im}(\alpha)>0$, and a real number $\beta$
such that for any $u,v \in V$, 
\[
\varphi(u)^* \varphi(v)=\begin{cases}
\alpha &\text{ if $(u,v) \in E$}, \\
\overline{\alpha} &\text{ if $(v,u) \in E$}, \\
\beta &\text{ otherwise}.
\end{cases}
\]
The image of the map $\varphi$ is called a {\it complex spherical representation} 
of $G$. 
 If two oriented graphs $G$ and $G'$ are not isomorphic, then representations of $G$ and $G'$ are not isomorphic. Let $\boldsymbol{A}$ be the adjacency matrix of $G$. 
 The Gram matrix $\boldsymbol{H}$ of 
a complex spherical representation of $G$ can be expressed by 
\[
\boldsymbol{H}=\boldsymbol{M}+c\sqrt{-1}(\boldsymbol{A}-\boldsymbol{A}^T),
\]
for some real number $c$ and some real matrix $\boldsymbol{M}$. 
Actually $\boldsymbol{M}$ is positive semidefinite.  
The matrix $\boldsymbol{M}$ can be identified with a real spherical representation of a simple graph $G'$ whose adjacency matrix is $\boldsymbol{A}+\boldsymbol{A}^T$. 
The dimension of a real spherical representation is studied in \cite{ES66,R10,NSpre}. 
Results related to real representations are helpful
to determine the dimension of a complex spherical representation. 
In this paper, 
we give an algorithm using only rational arithmetic to classify the largest $3$-codes in $\Omega(d)$.   
By the algorithm, we can classify the largest $3$-codes in $\Omega(d)$ for $d=1,2,3$.

This paper is organized as follows. 
In Section~\ref{sec:Erep},
we collect known results of Euclidean representations of a simple graph. 
In Section~\ref{sec:HM}, we show several results for Hermitian matrices that are used to determine the dimension of complex representation.  
In Section~\ref{sec:Crep}, we consider the dimension of a complex representation of an oriented graph.  
In Section~\ref{sec:algo}, we give an algorithm to classify the largest $3$-codes, and the largest $3$-codes in $\Omega(d)$ are classified for $d=1,2,3$ by computer calculation. 
In Section~\ref{sec:tight}, we show that there exists no tight 3-code except for $d= 1,2$.

\section{Euclidean representations of a simple graph} \label{sec:Erep}
In this section, we give several results for a real representation of a simple graph.  
Let $V$ be a finite set of order $n$, and $E\subset V \times V$. 
Let $G$ be a graph $(V,E)$. 
The {\it adjacency matrix} $\boldsymbol{A}$ of $G$ is the matrix indexed by $V$, with entries 
\[
\boldsymbol{A}_{xy}=
\begin{cases}
1 &\text{if $(x,y) \in E$}, \\
0 &\text{otherwise}.
\end{cases}
\]
Suppose $G$ is simple and $G$ is not a complete graph or a union of isolated vertices. 
Let $\boldsymbol{A}$ be the adjacency matrix of $G$, and $\overline{\boldsymbol{A}}$ that of the complement. 
The matrix $\boldsymbol{M}_c$  is defined to be 
\[
 \boldsymbol{M}_c=c \boldsymbol{A}+\overline{\boldsymbol{A}} 
\]
for a real number $c$ such that $0\leq c < 1$. 
A finite set $X$ in $\mathbb{R}^d$ is a {\it Euclidean representation} or a {\it real representation} of $G$
if the distance matrix of $X$ is $\boldsymbol{M}_c$ of $G$ for some $c$. 
 Let ${\rm Rep}(G)$ be the smallest integer $d$ such that a Euclidean representation of $G$ is in $\mathbb{R}^d$.
\begin{theorem}[\cite{ES66}] \label{thm:ES} 
Let $G$ be a simple graph.  Let $\boldsymbol{M}_c$ and ${\rm Rep}(G)$ be defined as above. 
Then there exists $\xi\in \mathbb{R}$ such that $0\leq \xi<1$ and the following hold. 
\begin{enumerate}
\item $\boldsymbol{M}_\xi$ is the distance matrix  in ${\rm Rep}(G)$ dimension. 
\item For $\xi< c< 1$, $\boldsymbol{M}_c$ is the distance matrix  in $n-1$ dimension, and not in $n-2$ dimension.
\item  For $0\leq c < \xi$, $\boldsymbol{M}_c$ is not a distance matrix in any dimension. 
\end{enumerate}
\end{theorem}

A Euclidean representation $X$ of $G$ is a {\it minimal representation} if the distance matrix of $X$ is $\boldsymbol{M}_\xi$, where $\xi$ is given in Theorem~\ref{thm:ES}.  
 Roy \cite{R10} determined ${\rm Rep}(G)$ by eigenvalues and eigenspaces of the adjacency matrix of $G$.   
Let $\boldsymbol{j}$ be the all-ones column vector. 
\begin{theorem}[{\cite[Lemmas~4,5,6, Theorem~7]{R10}}] \label{thm:dim}
Let $G$ be a simple graph with adjacency matrix $\boldsymbol{A}$. 
Let $\lambda_i$ be the $i$-th smallest 
distinct eigenvalue of $\boldsymbol{A}$, $m_i$ the multiplicity of $\lambda_i$, and $\mathcal{E}_i$ the eigenspace corresponding to $\lambda_i$. 
Let $\boldsymbol{P}_i$ be the orthogonal projection matrix onto $\mathcal{E}_i$. 
Let $\beta_i$ be the main angle of $\lambda_i$, namely, 
$
\beta_i= \sqrt{ (\boldsymbol{P}_i \cdot \boldsymbol{j})^T (\boldsymbol{P}_i \cdot \boldsymbol{j})/n}
$. Then the following hold:  
\begin{enumerate}
\item If $\beta_1=0$, then $\xi=(\lambda_1+1)/\lambda_1$ and ${\rm Rep}(G)=n-m_1-1$.
\item If $\beta_1 \ne 0$ and $m_1>1$, then $\xi=(\lambda_1+1)/\lambda_1$ and ${\rm Rep}(G)=n-m_1$.
\item If $\beta_2=0$, $m_1=1$, $\lambda_2<-1$,  and 
$
\beta_1^2/(\lambda_2-\lambda_1) = \sum_{i\geq 3} \beta_i^2/(\lambda_i-\lambda_2), 
$ 
then $\xi=(\lambda_2+1)/\lambda_2$ and ${\rm Rep}(G)=n-m_2-2$.
\item If $\beta_2=0$, $m_1=1$, $\lambda_2<-1$, and 
$
\beta_1^2/(\lambda_2-\lambda_1) > \sum_{i\geq 3} \beta_i^2/(\lambda_i-\lambda_2), 
$
then $\xi=(\lambda_2+1)/\lambda_2$ and ${\rm Rep}(G)=n-m_2-1$.
\item Otherwise, we have 
$\xi<(\lambda_1+1)/\lambda_1$, $\xi \ne (\lambda_2+1)/\lambda_2$ and ${\rm Rep}(G)=n-2$. 
\end{enumerate}
\end{theorem}
A graph $G$ is of {\it Type~$(i)$} if  $G$ satisfies condition $(i)$ from Theorem~\ref{thm:dim} for $i \in \{1,\ldots,5\}$. 
A Euclidean representation $X$ of $G$ is {\it spherical} if 
$X$ can be on a sphere.  
\begin{theorem}[\cite{NSpre}] \label{thm:spherical}
Let $G$ be a simple graph. Then the
 following hold.
 \begin{enumerate}
 \item If $G$ is of Type~$(1)$, $(2)$, or $(4)$, then the minimal representation of $G$ is spherical.
 \item If $G$ is of Type~$(3)$ or $(5)$, 
then the minimal representation of $G$ is not spherical. 
 \item A representation that satisfies condition $(2)$ from Theorem~$\ref{thm:ES}$ 
 is spherical.  
 \end{enumerate}
\end{theorem}

A symmetric matrix $\boldsymbol{M}$ is {\it dissimilarity} if each entry in $\boldsymbol{M}$ is non-negative, and each diagonal entry in $\boldsymbol{M}$ is zero.  
The smallest integer $d$ such that a dissimilarity matrix $\boldsymbol{M}$ is the distance matrix of some subset $X$ of $\mathbb{R}^d$ 
is called  
the {\it embedding dimension} of $\boldsymbol{M}$. 
Let $\boldsymbol{P}$ denote the square matrix of order $n$ defined by 
$\boldsymbol{P}=\boldsymbol{I}-(1/n)\boldsymbol{J}$, where $\boldsymbol{I}$ is the identity matrix and $\boldsymbol{J}$ is the all-ones matrix.   

\begin{lemma}[\cite{N81}] \label{rem:rank2}
If $\boldsymbol{M}$ is a dissimilarity matrix, then the following equivalent.  
\begin{enumerate}
\item  $\boldsymbol{M}$ is a distance matrix of embedding dimension $d$. 
\item $-\boldsymbol{P}\boldsymbol{M}\boldsymbol{P}$ is a positive semidefinite matrix of rank $d$. 
\end{enumerate}
\end{lemma}

\begin{lemma}[\cite{N81}] \label{rem:rank}
If $\boldsymbol{M}$ is a dissimilarity matrix, then the following are equivalent.
\begin{enumerate}
\item There uniquely exists $a \in \mathbb{R}$ such that  $a>0$, $-\boldsymbol{M}+a\boldsymbol{J}$ is a positive semidefinite matrix of rank $d$, $-\boldsymbol{M}+a'\boldsymbol{J}$ is a positive semidefinite matrix of rank $d+1$ for $a'>a$, 
and $-\boldsymbol{M}+c\boldsymbol{J}$ is not positive semidefinite for $c<a$.  
\item  $\boldsymbol{M}$ is the distance matrix of a subset of $S^{d-1}$, where $d$ is the embedding dimension of $\boldsymbol{M}$.  
\end{enumerate}
\end{lemma}

\section{Results on Hermitian matrices} 
\label{sec:HM}
In this section, we give several results for
Hermitian matrices that are used later.  
Let $\boldsymbol{H}$ be a Hermitian matrix of size $n$.
Let $\lambda$ be an eigenvalue of  $\boldsymbol{H}$. 
Let $\mathcal{E}$ be the eigenspace corresponding to $\lambda$. 
Let $\boldsymbol{P}_\lambda$ be the orthogonal projection matrix onto $\mathcal{E}$. 
Let $\boldsymbol{j}$ be the all-ones column vector. 
The {\it main angle} $\beta$ of $\lambda$ is defined to be
$
\beta=
 \sqrt{(\boldsymbol{P}_\lambda \cdot \boldsymbol{j}) ^\ast (\boldsymbol{P}_\lambda \cdot \boldsymbol{j})/n}.
$
Note that
$\beta=0$ if and only if $\mathcal{E} \subset \boldsymbol{j}^{\perp}$. 
An eigenvalue $\lambda$ is {\it main}  
if $\beta\ne 0$. 
Let $\boldsymbol{J}$ be the all-ones matrix, and   $\boldsymbol{I}$ the identity matrix. 
\begin{theorem}[\cite{NSudapre2}] \label{thm:interlace}
Let $\boldsymbol{H}$ be a Hermitian matrix, 
and $\boldsymbol{M}=\boldsymbol{H}+a \boldsymbol{J}$ for a real number $a$. 
Let $\tau_1,\ldots, \tau_r$ be the distinct main
eigenvalues of $\boldsymbol{H}$ such that $\tau_{1}<\tau_{2}<\cdots<\tau_{r}$. 
Let $\mu_1,\ldots, \mu_s$ be the distinct 
main eigenvalues of $\boldsymbol{M}$ such that $\mu_{1}<\mu_{2}<\cdots<\mu_{s}$.  
Let ${\beta}_{i}$ be the main angle of $\tau_i$.  
Then $r=s$ holds, and 
\begin{equation} \label{eq:f(x)}
\prod_{i=1}^r (\mu_{i}-x)= \prod_{i=1}^r(\tau_{i}-x) 
(1+a \sum_{j=1}^r \frac{n \beta_{j}^2}{\tau_{j}-x}). 
\end{equation} 
 Moreover, if $a>0$, then
$\tau_{1}<\mu_{1}<\tau_{2}< \cdots <\tau_{r}<\mu_{r}$, 
and if $a<0$, then
 $\mu_{1}<\tau_{1}<\mu_{2} <\cdots <\mu_{r}<\tau_{r}$.
\end{theorem}

\begin{lemma} \label{lem:1}
Let $\boldsymbol{H}$ be a Hermitian matrix of size $n$. 
Let $\tau_1,\ldots, \tau_r$ be the distinct main
eigenvalues of $\boldsymbol{H}$ such that $\tau_{1}<\tau_{2}<\cdots<\tau_{r}$. 
Let $\beta_i$ be the main angle of $\tau_i$. 
Let $\boldsymbol{P}$ be the orthogonal projection matrix onto $\boldsymbol{j}^{\perp}$, namely $\boldsymbol{P}=\boldsymbol{I}-(1/n) \boldsymbol{J}$. 
If $\boldsymbol{H}$ is not positive semidefinite,  
then the following are equivalent. 
\begin{enumerate}
\item There exists $a \in \mathbb{R}$ such that $a>0$ and $\boldsymbol{H}+ a \boldsymbol{J}$ is positive semidefinite. 
\item It follows that $\tau_{2}>0$, $\sum_{i=1}^r \beta_{i}^2/\tau_{i} <0$, and $\boldsymbol{P}\boldsymbol{H}\boldsymbol{P}$ is positive semidefinite. 
\end{enumerate}
Moreover, if $(1)$ holds, then $a \geq -1/(\sum_{i=1}^r n\beta_{i}^2/\tau_{i})$ holds.  
\end{lemma} 
 \begin{proof}
Let $\lambda$ be an eigenvalue of $\boldsymbol{H}$ that is not main. 
Let $\boldsymbol{v}$ be a normalized eigenvector corresponding to $\lambda$.
Note that $\boldsymbol{v}$ is orthogonal to the all-ones vector. 

 $(1) \Rightarrow (2)$: 
Since $\boldsymbol{H}+a\boldsymbol{J}$ is positive semidefinite,  
 we have 
$\lambda=\boldsymbol{v}^*\boldsymbol{H}\boldsymbol{v}=\boldsymbol{v}^*\boldsymbol{P}(\boldsymbol{H}+a\boldsymbol{J})\boldsymbol{P}\boldsymbol{v}\geq 0$.
 Since $\boldsymbol{H}$ is not positive semidefinite, 
 we have $\tau_{1}<0$. 
 Let $\mu_1 , \ldots, \mu_r$  
 be the 
 distinct main eigenvalues of $\boldsymbol{H}+a\boldsymbol{J}$ such that $ \mu_{1}<\mu_{2}<\cdots<\mu_{r}$. 
 By Theorem~\ref{thm:interlace},  
 we have $\tau_{1}< \mu_{1}<\tau_{2}$. 
Since $\boldsymbol{H}+a\boldsymbol{J}$ is positive semidefinite, we have $0\leq \mu_1<\tau_2$. 
By  equation \eqref{eq:f(x)} for $x=0$, 
it follows that $\sum_{i=1}^r n \beta_{i}^2/\tau_{i}<0$ and
 $a \geq -1/(\sum_{i=1}^r n \beta_{i}^2/\tau_{i})$.
 In particular,  $\mu_{1}= 0$ if and only if
 $a = -1/(\sum_{i=1}^r n \beta_{i}^2/\tau_{i})>0$.
  Since $\boldsymbol{H}+a\boldsymbol{J}$ is positive semidefinite, so is 
$\boldsymbol{P}(\boldsymbol{H}+a\boldsymbol{J})\boldsymbol{P}=\boldsymbol{P}\boldsymbol{H}\boldsymbol{P}$.

 $(2) \Rightarrow (1)$: 
Since $\boldsymbol{v}$ is orthogonal to the all-ones vector and $\boldsymbol{P}\boldsymbol{H}\boldsymbol{P}$ is positive semidefinite, 
we have 
\begin{align}\label{eq:lem1}
\lambda=\boldsymbol{v}^*\boldsymbol{H}\boldsymbol{v}=\boldsymbol{v}^*\boldsymbol{P}\boldsymbol{H}\boldsymbol{P}\boldsymbol{v}\geq 0.
\end{align}
Since $\boldsymbol{H}$ is not positive semidefinite, 
we have $\tau_1<0$.  
  By equation \eqref{eq:f(x)} for $x=0$ and $\tau_{2}>0$, 
 a matrix $\boldsymbol{H}+ a \boldsymbol{J}$ is positive semidefinite for $a \geq -1/(\sum_{i=1}^rn \beta_{i}^2/\tau_{i})>0$.
 \end{proof}
We can verify the following remarks by the proof of Lemma~\ref{lem:1}. 
\begin{remark} \label{rem:Rank(H)}
If Lemma~\ref{lem:1} (1) holds, then
\begin{enumerate}
 \item ${\rm Rank}(\boldsymbol{H}+ a \boldsymbol{J})={\rm Rank}(\boldsymbol{H})-1$ for $a = -1/(\sum_{i=1}^r n\beta_{i}^2/\tau_{i})$,
 \item ${\rm Rank}(\boldsymbol{H}+ a \boldsymbol{J})={\rm Rank}(\boldsymbol{H})$ for $a > -1/(\sum_{i=1}^r n\beta_{i}^2/\tau_{i})$.
\end{enumerate}
\end{remark}
\begin{remark} \label{rem:null}
If Lemma~\ref{lem:1} (2) holds, then the null space of 
$\boldsymbol{H}$ is contained in $\boldsymbol{j}^{\perp}$. 
\end{remark}

\begin{remark} \label{lem:Rank(PHP)}
If Lemma~\ref{lem:1} $(2)$ holds,  
then ${\rm Rank} (\boldsymbol{H}+a\boldsymbol{J})={\rm Rank}(\boldsymbol{P}\boldsymbol{H}\boldsymbol{P})$ for $a = -1/(\sum_{i=1}^r n\beta_{i}^2/\tau_{i})$. 
\end{remark}


\begin{theorem} \label{thm:E_0}
Let $\boldsymbol{H}$ be a 
Hermitian matrix.  
Let $\boldsymbol{M}$ and $\boldsymbol{A}$ be the real matrices such that $\boldsymbol{H}=\boldsymbol{M}+ \sqrt{-1} \boldsymbol{A}$. 
Let $\mathcal{E}_0$ be the null space of $\sqrt{-1} \boldsymbol{A}$.
Let $\mathcal{E}_0'$ be the null space of $\boldsymbol{M}$. 
If $\boldsymbol{H}$ is positive semidefinite, then $\mathcal{E}_0' \subseteq \mathcal{E}_0$ holds.  
\end{theorem}
\begin{proof}
Since $\boldsymbol{M}$ is a real symmetric matrix, we can take a basis of $\mathcal{E}_0'$ consisting of real vectors. 
For a real vector $\boldsymbol{v} \in \mathcal{E}_0'$, we have
\[
\boldsymbol{v}^* \boldsymbol{H} \boldsymbol{v}= \boldsymbol{v}^* \boldsymbol{M} \boldsymbol{v}+ \sqrt{-1}  \boldsymbol{v}^* \boldsymbol{A} \boldsymbol{v}=0  
\]
because $\boldsymbol{A}$ is skew-symmetric.    
Since $\boldsymbol{H}$ is a positive semidefinite, $\boldsymbol{v}^* \boldsymbol{H} \boldsymbol{v}=0$ if and only if $\boldsymbol{H} \boldsymbol{v}=\boldsymbol{o}$. 
It thus follows that  
\[
\boldsymbol{o}=\boldsymbol{H} \boldsymbol{v}=\boldsymbol{M} \boldsymbol{v}+  \sqrt{-1}  \boldsymbol{A} \boldsymbol{v} =  
\sqrt{-1}   \boldsymbol{A} \boldsymbol{v}. 
\]
Therefore $\mathcal{E}_0' \subseteq \mathcal{E}_0$ holds.
\end{proof}

\begin{theorem} \label{thm:half}
Let $\boldsymbol{H}$ be a 
Hermitian matrix.  
Let $\boldsymbol{M}$ and $\boldsymbol{A}$ be the real matrices such that $\boldsymbol{H}=\boldsymbol{M}+ \sqrt{-1} \boldsymbol{A}$. 
If $\boldsymbol{H}$ is positive semidefinite,  
then $2 {\rm Rank}(\boldsymbol{H}) \geq {\rm Rank}({\boldsymbol{M}})$.  
\end{theorem}
\begin{proof}
By Theorem~\ref{thm:E_0}, we have $\mathcal{E}_0' \subseteq \mathcal{E}_0$. 
Let $\mathcal{E}_+$ ({\it resp}. $\mathcal{E}_-$) be the direct sum of eigenspaces corresponding
to the positive ({\it resp}. negative) eigenvalues of $\sqrt{-1}\boldsymbol{A}$. 
It is easily proved that $\dim \mathcal{E}_+ =\dim \mathcal{E}_-$. 
For a non-zero vector $\boldsymbol{v} \in \mathcal{E}_+ \oplus ((\mathcal{E}_0')^{\perp} \cap \mathcal{E}_0)$, we have
$\boldsymbol{v}^*\boldsymbol{H}\boldsymbol{v}>0$ because $\boldsymbol{M}$ is positive semidefinite.
Therefore,
\begin{align*}
{\rm Rank} (\boldsymbol{H}) &\geq \dim (\mathcal{E}_+ \oplus ((\mathcal{E}_0')^{\perp} \cap \mathcal{E}_0))\\
&= \dim (\mathcal{E}_+) + \dim ((\mathcal{E}_0')^{\perp} \cap \mathcal{E}_0)\\
&= \dim (\mathcal{E}_+) + \dim ((\mathcal{E}_0')^{\perp})+\dim (\mathcal{E}_0)-\dim ((\mathcal{E}_0')^{\perp}+\mathcal{E}_0)\\
& =\frac{1}{2} {\rm Rank} (\boldsymbol{A}) + {\rm Rank}(\boldsymbol{M})+(n- {\rm Rank}(\boldsymbol{A}))-n\\
&={\rm Rank} (\boldsymbol{M}) -\frac{1}{2} {\rm Rank}(\boldsymbol{A}) \\
& \geq {\rm Rank} (\boldsymbol{M}) -\frac{1}{2} {\rm Rank}(\boldsymbol{M}) \\
& =\frac{1}{2}{\rm Rank}(\boldsymbol{M}),
\end{align*}
where $n$ is the size of $\boldsymbol{H}$. 
Thus the theorem follows. 
\end{proof}

\begin{theorem}\label{thm:eta}
Let $\boldsymbol{H}$ be a Hermitian matrix. 
Let $\boldsymbol{M}$ and $\boldsymbol{A}$ be the real matrices such that $\boldsymbol{H}=\boldsymbol{M}+ \sqrt{-1} \boldsymbol{A}$. 
Let $\mathcal{E}_0$ be the null space of $\sqrt{-1} \boldsymbol{A}$. 
Let $\mathcal{E}_0'$ be the null space of $\boldsymbol{M}$. 
Suppose $\boldsymbol{M}$ is positive semidefinite, and 
$\mathcal{E}_0' \subseteq \mathcal{E}_0$ holds. 
Then there uniquely exists $\eta >0$ such that the following hold:
\begin{enumerate}
\item $\boldsymbol{M} +\eta \sqrt{-1}\boldsymbol{A}$ is 
positive semidefinite, 
and its rank is smaller than ${\rm Rank} (\boldsymbol{M})$. 
\item  $\boldsymbol{M} +c \sqrt{-1}\boldsymbol{A}$ is 
positive semidefinite for $0\leq c < \eta$, 
and its rank is equal to ${\rm Rank} (\boldsymbol{M})$. 
\item  $\boldsymbol{M} +c\sqrt{-1} \boldsymbol{A}$ is not
positive semidefinite for $\eta<c$. 
\end{enumerate}
\end{theorem}
\begin{proof}
Let $\Phi(c)$ be the function defined by 
\[
\Phi(c):=\min_{\boldsymbol{v} \in (\mathcal{E}_0')^{\perp}, \boldsymbol{v}^*\boldsymbol{v}=1} \boldsymbol{v}^*(
\boldsymbol{M} +c \sqrt{-1}\boldsymbol{A}) \boldsymbol{v}.
\]
Note that $\Phi(c) \geq 0$ if and only if $\boldsymbol{M} +c \sqrt{-1}\boldsymbol{A}$ is positive semidefinite, and 
${\rm Rank}(\boldsymbol{M} +c \sqrt{-1}\boldsymbol{A}) \leq {\rm Rank} (\boldsymbol{M})$. 
In particular, 
 $\Phi(c) = 0$ if and only if 
${\rm Rank}(\boldsymbol{M} +c \sqrt{-1}\boldsymbol{A}) < {\rm Rank} (\boldsymbol{M})$.
Since $\Phi(c)$ is the minimum value of the collection of linear functions in $c$, 
the function $\Phi(c)$ is concave. 
Since $\boldsymbol{M}$ is positive semidefinite, we have $\Phi(0)>0$.
There exists $\boldsymbol{v} \in (\mathcal{E}_0')^{\perp}$ such that 
$\boldsymbol{v}^* (\sqrt{-1}\boldsymbol{A}) \boldsymbol{v} < 0$.
It therefore follows that  $\lim_{c\rightarrow \infty } \Phi(c)=-\infty$.
By the intermediate value theorem, this theorem follows. 
\end{proof}

\section{Representations of an oriented graph} \label{sec:Crep}

Let $X$ be a complex spherical $3$-code with angle set 
$D(X)= \{\alpha, \overline{\alpha}, \beta\}$, where $\alpha$ is an imaginary number with ${\rm Im}(\alpha)>0$, and $\beta \in \mathbb{R}$.   
Let $E=\{(\boldsymbol{x},\boldsymbol{y}) \in X \times X \mid  \boldsymbol{x}^* \boldsymbol{y} = \alpha  \}$, and $E'=\{(\boldsymbol{x},\boldsymbol{y}) \mid (\boldsymbol{x},\boldsymbol{y}) \in E \text{ or } (\boldsymbol{y},\boldsymbol{x}) \in E \}$. 
 Let $G$ be the oriented graph $(X,E)$ with adjacency matrix $\boldsymbol{A}$.  
Let $G'$ be the simple graph $(X,E')$ with adjacency matrix $\boldsymbol{B}$.  
Let $\overline{\boldsymbol{B}}$ be the adjacency matrix of the complement of $G'$. 
The Gram matrix $\boldsymbol{H}$ of 
a complex spherical representation of $G$ can be expressed by 
\[
\boldsymbol{H}=\boldsymbol{M}+c\sqrt{-1}(\boldsymbol{A}-\boldsymbol{A}^T)
\]
for  a real number $c$ and a real matrix  $\boldsymbol{M}$.  
Let $\phi$ be a map from $\Omega(d)$ 
to $S^{2d-1}$ defined by 
\[
\phi(u_1+v_1 \sqrt{-1}, \ldots , u_d+v_d \sqrt{-1})=(u_1,v_1,\ldots,u_d,
v_d).
\]
Note that $\phi(\boldsymbol{x})^T \phi(\boldsymbol{y})={\rm Re}(\boldsymbol{x}^* \boldsymbol{y})$ for $\boldsymbol{x},\boldsymbol{y} \in \Omega(d)$. 
The matrix $\boldsymbol{M}$ is the Gram matrix of $\phi(X)=\{\phi(\boldsymbol{x})\mid \boldsymbol{x} \in X\}$. The representation $\phi(X)$ of $G'$ is spherical.  By Lemma~\ref{rem:rank}, 
$\boldsymbol{M}$ can be expressed by 
\[\boldsymbol{M}=-(b\boldsymbol{B}+\overline{\boldsymbol{B}})+a\boldsymbol{J}
\]
for $a>0$ and $b\geq 0$. Note that $b\boldsymbol{B}+\overline{\boldsymbol{B}}$ is the distance matrix of $\phi(X)$ after rescaling  the two distances to $1$ and $b$.   
Since $\phi(X)$ is spherical, $\phi(X)$ is the minimal representation of $G'$ of Type~(1), (2)
or (4), or a non-minimal representation by Theorem~\ref{thm:spherical}.

By Theorem~\ref{thm:E_0}, the null space $\mathcal{E}_0'$ of 
$\boldsymbol{M}$ must be contained in the null space $\mathcal{E}_0$ 
of $\sqrt{-1}(\boldsymbol{A}-\boldsymbol{A}^T)$. 
When we consider a minimal-dimensional representation of a given oriented graph $G$, 
 the minimal representation of $G'$ rarely satisfies $\mathcal{E}_0' \subseteq \mathcal{E}_0$.  
We give simple examples: 
\[
G_1: \boldsymbol{A}_1=\left(
\begin{array}{cccc}
 0 & 1 & 0 & 0 \\
 0 & 0 & 1 & 0 \\
 0 & 0 & 0 & 1 \\
 1 & 0 & 0 & 0
\end{array}
\right), 
\qquad
G_2: \boldsymbol{A}_2=
\left(
\begin{array}{cccc}
 0 & 0 & 0 & 0 \\
 1 & 0 & 1 & 0 \\
 0 & 0 & 0 & 1 \\
 1 & 0 & 0 & 0
\end{array}
\right).
\]
Then both $G_1'$ and $G_2'$ are 
the cycle $C_4$. Indeed $C_4$ is of Type~(1), and 
its minimal representation is the vertex set of the square 
in $\mathbb{R}^2$. 
The Gram matrix of the square can be expressed by 
\[
\boldsymbol{M}_1=-(\frac{1}{2}\boldsymbol{B}+\overline{\boldsymbol{B}})+\frac{1}{2}\boldsymbol{J}=
\left(
\begin{array}{cccc}
 \frac{1}{2} & 0 & -\frac{1}{2} & 0 \\
 0 & \frac{1}{2} & 0 & -\frac{1}{2} \\
 -\frac{1}{2} & 0 & \frac{1}{2} & 0 \\
 0 & -\frac{1}{2} & 0 & \frac{1}{2}
\end{array}
\right).
\]
The null space of $\boldsymbol{M}_1$ is ${\rm Span}\{(1,0,1,0),(0,1,0,1)\}$. 
This coincides with the null space of $\sqrt{-1}(\boldsymbol{A}_1-\boldsymbol{A}_1^T)$. 
Actually we can give a minimal-dimensional representation in $\Omega(1)$
of $G_1$
as 
\[
\boldsymbol{H}_1=-(\frac{1}{2}\boldsymbol{B}+\overline{\boldsymbol{B}})+\frac{1}{2}\boldsymbol{J}+\frac{1}{2}\sqrt{-1}(\boldsymbol{A}_1-\boldsymbol{A}_1^T)
=\left(
\begin{array}{cccc}
 \frac{1}{2} & \frac{\sqrt{-1}}{2} & -\frac{1}{2} & -\frac{\sqrt{-1}}{2} \\
 -\frac{\sqrt{-1}}{2} & \frac{1}{2} & \frac{\sqrt{-1}}{2} & -\frac{1}{2} \\
 -\frac{1}{2} & -\frac{\sqrt{-1}}{2} & \frac{1}{2} & \frac{\sqrt{-1}}{2} \\
 \frac{\sqrt{-1}}{2} & -\frac{1}{2} & -\frac{\sqrt{-1}}{2} & \frac{1}{2}
\end{array}
\right).
\]
On the other hand, the eigenvalues of $\sqrt{-1}(\boldsymbol{A}_2-\boldsymbol{A}_2^T)$ are 
$\{-\sqrt{2}, -\sqrt{2}, \sqrt{2}, \sqrt{2}\}$, and hence the null
space is empty. In this case, ${\rm Rank}(\boldsymbol{M}_2)$ must be 4, and we use a non-minimal representation
of $G'$: 
\[
\boldsymbol{M}_2=-(\boldsymbol{B}+\overline{\boldsymbol{B}})+\boldsymbol{J}=\left(
\begin{array}{cccc}
 1 & 0 & 0 & 0 \\
 0 & 1 & 0 & 0 \\
 0 & 0 & 1 & 0 \\
 0 & 0 & 0 & 1
\end{array}
\right). 
\]  
Then we can give a minimal-dimensional representation in $\Omega(2)$ 
of $\boldsymbol{A}_2$ as 
\[
\boldsymbol{H}_2=-(\boldsymbol{B}+\overline{\boldsymbol{B}})+\boldsymbol{J}+\sqrt{\frac{-1}{2}}(\boldsymbol{A}_2-\boldsymbol{A}_2^T)
=\left(
\begin{array}{cccc}
 1 & -\sqrt{\frac{-1}{2}} & 0 & -\sqrt{\frac{-1}{2}} \\
 \sqrt{\frac{-1}{2}} & 1 & \sqrt{\frac{-1}{2}} & 0 \\
 0 & -\sqrt{\frac{-1}{2}} & 1 & \sqrt{\frac{-1}{2}} \\
 \sqrt{\frac{-1}{2}} & 0 & -\sqrt{\frac{-1}{2}} & 1
\end{array}
\right). 
\]


The dimension of a non-minimal representation $X'$ of a simple graph $G'$ is $n-1$, where $n$ is the order of $G'$.   
If $X'$ is used in order to give a representation $X$ of an oriented graph $G$, 
then the dimension $d$ of  $X$ is at least $(n-1)/2$ by Theorem~\ref{thm:half}, namely $n \leq 2d+1$.   
The union of $d$ triangles that are orthogonal to each other is a spherical $3$-code in $\Omega(d)$ and has size $3d$.    
Therefore it is enough to consider a representation $X$ of $G$ obtained from the minimal representation of $G'$ in order to determine the largest $3$-codes.  

We consider the minimal-dimensional 
representation of  $G$ obtained from the minimal 
representation of $G'$.  
Throughout this section, we suppose 
$G'$ has non-zero $\boldsymbol{B}$ and $\overline{\boldsymbol{B}}$, and $G'$ is of Type~(1), (2), or (4).
 Let $\boldsymbol{H}(a,c)$ denote the matrix defined by 
\begin{equation} \label{eq:H(a,c)}
\boldsymbol{H}(a,c)=-(\xi \boldsymbol{B}+ \overline{\boldsymbol{B}})+ a\boldsymbol{J}+c\sqrt{-1}(\boldsymbol{A}-\boldsymbol{A}^T)
\end{equation}
for real numbers $a$ and $c$, where $\xi$ is the positive number given in Theorem~\ref{thm:ES}. 
Note that $\xi \boldsymbol{B}+ \overline{\boldsymbol{B}}$ be the distance matrix of the minimal representation of $G'$.  
We would like to determine $a$ and $c$ so that $a>0$, $c>0$,  
$\boldsymbol{H}(a,c)$ is positive semidefinite, and the rank of $\boldsymbol{H}(a,c)$
is minimal.  
Let $\mathcal{E}_0$ be the null space 
of $\sqrt{-1}(\boldsymbol{A}-\boldsymbol{A}^T)$, and $\mathcal{E}_0'$ be that of $-(\xi \boldsymbol{B}+ \overline{\boldsymbol{B}})$. 
\begin{remark} \label{rem:nullB}
If $G'$ is of Type~(1), (2), or (4), then $\mathcal{E}_0' \subset \boldsymbol{j}^{\perp}$
holds by Lemma~\ref{rem:rank} and Remark~\ref{rem:null}.
\end{remark}

Since  the diagonal entries in $\boldsymbol{H}(0,c)$ are zero,    $\boldsymbol{H}(0,c)$ is not a positive semidefinite.  
If $\boldsymbol{H}(a,c)$ is positive semidefinite, then 
$\boldsymbol{H}(0,c)$ satisfies condition (2) from  Lemma~\ref{lem:1}, and hence
$\boldsymbol{P}\boldsymbol{H}(0,c)\boldsymbol{P}$ is positive semidefinite. 
If $\boldsymbol{H}(0,c)$ satisfies condition (2) from  Lemma~\ref{lem:1},  
then there uniquely exists a positive number $a$ such that ${\rm Rank}(\boldsymbol{H}(a,c))$ is minimal, and ${\rm Rank}(\boldsymbol{H}(a,c))={\rm Rank}(\boldsymbol{P}\boldsymbol{H}(0,c)\boldsymbol{P})$
by Remarks~\ref{rem:Rank(H)} and \ref{lem:Rank(PHP)}. 
Therefore we would like to choose $c$ so that $\boldsymbol{P}\boldsymbol{H}(0,c)\boldsymbol{P}$ is positive
semidefinite, and ${\rm Rank}(\boldsymbol{P}\boldsymbol{H}(0,c)\boldsymbol{P})$ is minimal.  
The following lemma shows such possible $c$ and the 
evaluation of ${\rm Rank}(\boldsymbol{P}\boldsymbol{H}(0,c)\boldsymbol{P})$. 

\begin{lemma}\label{lem:eta2}
Let $G$ be an oriented graph $(V,E)$ with adjacency matrix $\boldsymbol{A}$. 
Let $G'$ be the simple graph $(V, E')$ with adjacency matrix $\boldsymbol{B}$, where $E'=\{(u,v) \mid (u,v) \in E \text{ or } (v,u) \in E \}$.  
Let $\overline{\boldsymbol{B}}$ be the adjacency matrix of the complement of $G'$. 
  Let $\boldsymbol{H}(a,c)$ be the matrix defined by 
\[
\boldsymbol{H}(a,c)=-(\xi \boldsymbol{B}+ \overline{\boldsymbol{B}})+ a\boldsymbol{J}+c\sqrt{-1}(\boldsymbol{A}-\boldsymbol{A}^T)
\]
for real numbers $a$ and $c$, where $\xi$ is the positive number given in Theorem~\ref{thm:ES}. 
Let $\mathcal{E}_0$ be the null space 
of $\sqrt{-1}(\boldsymbol{A}-\boldsymbol{A}^T)$. 
Let $\mathcal{E}_0'$ be the null space of $-(\xi \boldsymbol{B}+ \overline{\boldsymbol{B}})$. 
If $\mathcal{E}_0' \subseteq \mathcal{E}_0$ holds, 
then there uniquely exists a positive number $\eta$ such that
\begin{enumerate}
\item $\boldsymbol{P}\boldsymbol{H}(0,\eta)\boldsymbol{P}$ is positive semidefinite, and 
\[{\rm Rank} (\boldsymbol{P}\boldsymbol{H}(0,\eta)\boldsymbol{P})< {\rm Rank}(\boldsymbol{P}\boldsymbol{H}(0,0)\boldsymbol{P}),\] 
\item $\boldsymbol{P}\boldsymbol{H}(0,c)\boldsymbol{P}$ is positive semidefinite, and 
\[{\rm Rank} (\boldsymbol{P}\boldsymbol{H}(0,c)\boldsymbol{P})={\rm Rank} (\boldsymbol{P}\boldsymbol{H}(0,0)\boldsymbol{P})\]
 for $0< c< \eta$, 
\item $\boldsymbol{P}\boldsymbol{H}(0,c)\boldsymbol{P}$ is not positive semidefinite for $\eta<c$.
\end{enumerate}
\end{lemma}
\begin{proof}
It follows that 
\[
\boldsymbol{P}\boldsymbol{H}(0,c)\boldsymbol{P}=-\boldsymbol{P}(\xi \boldsymbol{B}+ \overline{\boldsymbol{B}})\boldsymbol{P}+c\sqrt{-1}\boldsymbol{P}(\boldsymbol{A}-\boldsymbol{A}^T)\boldsymbol{P}. 
\]
It is easily shown that the null space of $-\boldsymbol{P}(\xi \boldsymbol{B}+ \overline{\boldsymbol{B}})\boldsymbol{P}$
is contained in that of $\sqrt{-1}\boldsymbol{P}(\boldsymbol{A}-\boldsymbol{A}^T)\boldsymbol{P}$. 
This lemma follows from Theorem~\ref{thm:eta}.  
\end{proof}

Next we have to check whether $\boldsymbol{H}(0,c)$ satisfies 
 condition  (2) from Lemma~\ref{lem:1} for $0< c \leq \eta$, where $\eta$ is the positive number given in Lemma~\ref{lem:eta2}. 
If  $\boldsymbol{H}(0,c)$ satisfies 
condition (2) from Lemma~\ref{lem:1}, we can construct a representation of $G$
by choosing suitable number $a$. 
\begin{theorem} \label{thm:3-code} 
Let $G$ be an oriented graph $(V,E)$ with adjacency matrix $\boldsymbol{A}$. 
Let $G'$ be the simple graph $(V, E')$ with adjacency matrix $\boldsymbol{B}$, where $E'=\{(u,v) \mid (u,v) \in E \text{ or } (v,u) \in E \}$. 
Suppose $G'$ is of Type (1), (2), or (4).  
Let $\overline{\boldsymbol{B}}$ be the adjacency matrix of the complement of $G'$. 
  Let $\boldsymbol{H}(a,c)$ be the matrix defined by 
\[
\boldsymbol{H}(a,c)=-(\xi \boldsymbol{B}+ \overline{\boldsymbol{B}})+ a\boldsymbol{J}+c\sqrt{-1}(\boldsymbol{A}-\boldsymbol{A}^T)
\]
for real numbers $a$ and $c$, where $\xi$ is the positive number given in Theorem~\ref{thm:ES}.
Let \[U=\{(a,c) \mid \text{$\boldsymbol{H}(a,c)$ is positive semidefinite, $a>0$, $c>0$} \},\] 
and 
\[{\rm Rep}(G)=\min\{{\rm Rank}(\boldsymbol{H}(a,c))  \mid (a,c) \in U \}.\]  
Let ${\rm Rep}(G')$ be the dimension of the minimal representation of $G'$. 
Let $\mathcal{E}_0$ be the null space 
of $\sqrt{-1}(\boldsymbol{A}-\boldsymbol{A}^T)$. 
Let $\mathcal{E}_0'$ be the null space of $-(\xi \boldsymbol{B}+ \overline{\boldsymbol{B}})$. 
Let $\eta$ be a positive number given in Lemma~{\rm \ref{lem:eta2}}. 
If $\mathcal{E}_0' \subseteq \mathcal{E}_0$ holds,   
then the following hold. 
\begin{enumerate}
\item If $\boldsymbol{H}(0,\eta)$ satisfies condition $(1)$ from Lemma~{\rm \ref{lem:1}}, then  
\[{\rm Rep}(G)={\rm Rank}(\boldsymbol{H}(0,\eta))-1 <{\rm Rep}(G').\]
\item If $\boldsymbol{H}(0,\eta)$ does not satisfy condition $(1)$ from Lemma~{\rm \ref{lem:1}}, then  
\[{\rm Rep}(G)={\rm Rank}(\boldsymbol{H}(0,0))-1={\rm Rep}(G').\]
\end{enumerate}
\end{theorem}

\begin{proof}
Since the minimal representation of $G'$ is spherical, 
there uniquely exists $a' \in \mathbb{R}$ 
such that 
$\boldsymbol{H}(a',0)$ is positive semidefinite and
${\rm Rep}(G')={\rm Rank}(\boldsymbol{H}(a',0))$ by Lemma~\ref{rem:rank}. By Remark~\ref{lem:Rank(PHP)}, 
it follows that ${\rm Rank}(\boldsymbol{H}(a',0))=
{\rm Rank}(\boldsymbol{P}\boldsymbol{H}(0,0)\boldsymbol{P})$, and hence 
\begin{equation} \label{eq:pr-1}
{\rm Rep}(G')=
{\rm Rank}(\boldsymbol{P}\boldsymbol{H}(0,0)\boldsymbol{P}).
\end{equation}

Since $\boldsymbol{H}(a,c)$ is positive semidefinite
for each $(a,c) \in U$, 
the matrix $\boldsymbol{P}\boldsymbol{H}(0,c)\boldsymbol{P}$, which is equal to $\boldsymbol{P}\boldsymbol{H}(a,c)\boldsymbol{P}$, is positive semidefinite. 
Since $\boldsymbol{P}\boldsymbol{H}(0,c)\boldsymbol{P}$
is positive semidefinite and $\mathcal{E}_0' \subseteq \mathcal{E}_0$, it follows that $0<c\leq \eta$,  
\begin{equation} \label{eq:pr-2}
{\rm  Rank}(\boldsymbol{P}\boldsymbol{H}(0,c)\boldsymbol{P}) = {\rm  Rank}(\boldsymbol{P}\boldsymbol{H}(0,0)\boldsymbol{P}) 
\end{equation}
for $0<c<\eta$, and  
\begin{equation} \label{eq:pr-2_2}
{\rm  Rank}(\boldsymbol{P}\boldsymbol{H}(0,\eta)\boldsymbol{P}) < {\rm  Rank}(\boldsymbol{P}\boldsymbol{H}(0,0)\boldsymbol{P})  
\end{equation}
for $c=\eta$ by Lemma~\ref{lem:eta2}. 

If $\boldsymbol{H}(a,c)$ is positive semidefinite, 
then there uniquely exists $a_c \in \mathbb{R}$ 
such that
$\boldsymbol{H}(a_c,c)$ is positive semidefinite and
\begin{equation} \label{eq:pr-3}
{\rm Rank}(\boldsymbol{P}\boldsymbol{H}(0,c)\boldsymbol{P})=
{\rm Rank}(\boldsymbol{H}(a_c,c))=
{\rm Rank}(\boldsymbol{H}(0,c))-1
 \leq {\rm Rank}(\boldsymbol{H}(a,c))
\end{equation}
 by Remark~\ref{rem:Rank(H)} and Remark~\ref{lem:Rank(PHP)}. 

(1):  Since $\boldsymbol{H}(0,\eta)$ satisfies condition $(1)$ from Lemma~{\rm \ref{lem:1}}, 
there exists $a \in \mathbb{R}$ such that $(a,\eta) \in U$.  
From equations \eqref{eq:pr-2},  \eqref{eq:pr-2_2} and \eqref{eq:pr-3}, 
for each $(a,c) \in U$ with $c\ne \eta$, 
\begin{align}
{\rm Rank}(\boldsymbol{H}(0,\eta))-1&={\rm Rank}(\boldsymbol{H}(a_\eta,\eta))
={\rm Rank}(\boldsymbol{P}\boldsymbol{H}(0,\eta)\boldsymbol{P}) \nonumber \\
&<{\rm Rank}(\boldsymbol{P}\boldsymbol{H}(0,0)\boldsymbol{P})=
{\rm Rank}(\boldsymbol{P}\boldsymbol{H}(0,c)\boldsymbol{P}) \nonumber \\
&={\rm Rank}(\boldsymbol{H}(a_c,c)) 
\leq {\rm Rank}(\boldsymbol{H}(a,c)). \label{eq:pr-4}
\end{align}
For $(a,\eta) \in U$, 
\begin{equation}
{\rm Rank}(\boldsymbol{H}(0,\eta))-1={\rm Rank}(\boldsymbol{H}(a_\eta,\eta))
\leq {\rm Rank}(\boldsymbol{H}(a,\eta))  \label{eq:pr-5}
\end{equation}
by equation \eqref{eq:pr-3}. 
The assertion follows form equations \eqref{eq:pr-1}, \eqref{eq:pr-4}, and \eqref{eq:pr-5}.

(2): 
Since the minimal representation of $G'$ is spherical, 
there exists $a' \in \mathbb{R}$ such that $\boldsymbol{H}(a',0)$ is 
positive semidefinite. 
Since $\mathcal{E}_0' \subset \boldsymbol{j}^{\perp}$ by Remark~\ref{rem:nullB}, the null space of 
$\boldsymbol{H}(a',0)$ is also $\mathcal{E}_0'$.  
By Theorem~\ref{thm:eta}, there exists 
a positive number $\eta'$ such that $0<\eta'<\eta$ and $\boldsymbol{H}(a',\eta')$ is positive semidefinite.
For each $(a,c) \in U$, it follows from equations \eqref{eq:pr-2} and \eqref{eq:pr-3} that 
\begin{multline}\label{eq:pr-6}
{\rm Rank}(\boldsymbol{H}(a_{\eta'},\eta'))={\rm Rank}(\boldsymbol{P}\boldsymbol{H}(0,\eta')\boldsymbol{P})
={\rm Rank}(\boldsymbol{P}\boldsymbol{H}(0,0)\boldsymbol{P})\\
={\rm Rank}(\boldsymbol{P}\boldsymbol{H}(0,c)\boldsymbol{P}) \leq{\rm Rank}(\boldsymbol{H}(a,c)).
\end{multline}
It follows from Lemma~\ref{rem:rank2} and Remark \ref{rem:Rank(H)} that 
\begin{equation} \label{eq:pr-7}
{\rm Rank}(\boldsymbol{P}\boldsymbol{H}(0,0)\boldsymbol{P})={\rm Rank}(\boldsymbol{H}(0,0))-1. 
\end{equation}
The assertion follows from equations \eqref{eq:pr-1}, \eqref{eq:pr-6}, and \eqref{eq:pr-7}.  
\end{proof}

\section{Algorithm to give the largest $3$-codes} \label{sec:algo}
In this section, we give an {\it algorithm} using only rational arithmetic to classify the largest $3$-codes in 
$\Omega(d)$ for given dimension $d$. First we collect several algorithms used in the {\it algorithm}.  
An interval $[a,b]$ is an {\it isolating interval} for a polynomial $f$ and a real number $\gamma$ such that $f(\gamma)=0$ 
if $a$ and $b$ are rational numbers, $a<\gamma<b$,  and $[a,b]$ contains no other roots of $f$.  
A real algebraic number $\gamma$ is represented by a pair $(f_\gamma,I)$, where $f_\gamma$ is the minimal polynomial 
of $\gamma$ over the field of rationals, and $I$ is an isolating interval $[a,b]$ for $f$ and $\gamma$.   
If $f$ is the minimal polynomial of $\gamma$, then $\gamma$ is a simple root and an isolating interval 
$[a,b]$ satisfies $f(a)f(b)< 0$. Since we have an explicit lower bound for the separation of roots of an integral polynomial \cite{R79}, we easily obtain the isolating interval $[a,b]$.  

\begin{lemma}[{\cite{L97}}] \label{lem:signP}
There is an algorithm (using only rational arithmetic) which takes as input an algebraic number $\gamma$ 
and a polynomial $f$ with integer coefficients, and determines the sign of the number $f(\gamma)$.  
\end{lemma}
\begin{proof}
Let $g_\gamma$ be the minimal polynomial of $\gamma$ over $\mathbb{Q}$.  
Since $g_\gamma$ is irreducible, $f(\gamma)=0$ if and only if $g_\gamma$ divides $f$. 
Suppose $g_\gamma$ does not divide $f$. 
We can find an isolating interval $[a,b]$ for $g_\gamma$ and 
$\gamma$, such that $[a,b]$ contains 
no root of $f$. 
Then the sign of $f(a)$ is equal to that of $f(\gamma)$.    
\end{proof}

\begin{lemma} \label{lem:rankM}
There is an algorithm (using only rational arithmetic) which takes as 
input an real algebraic number $\gamma$ and a symmetric matrix $\boldsymbol{M}(t)$ whose entries are in $\mathbb{Q}[t]$, 
and determines the number of the positive eigenvalues and the number of the negative eigenvalues of  $\boldsymbol{M}(\gamma)$. This 
decides whether 
$\boldsymbol{M}(\gamma)$  is positive semidefinite. 
\end{lemma}

\begin{proof}
Let $P(t,x)$ be the polynomial defined by 
\[
P(t,x)=| \boldsymbol{M}(t)-x \boldsymbol{I} |. 
\]
Let $P_i(t)$ be the coefficient of $x^i$ in $P(x)=P(t,x)$.  
By Lemma~\ref{lem:signP}, we can determine the sign of $P_i(\gamma)$. 
Using Descartes' rule of signs, 
the number of the positive roots and the number of the negative roots  of $P(x)=P(\gamma,x)$ 
are determined by the list of the signs of $P_i(\gamma)$.   
\end{proof}

Let $f$ be an irreducible polynomial over $\mathbb{Q}(\gamma)$ for an algebraic integer $\gamma$. 
Let $\eta$ be a zero of $f$. Using Sturm's theorem, $\eta$ can be represented by $(f, I)$, where $I$ is an 
isolating interval for $f$ and $\eta$.   Here the signs in Sturm's sequence can be determined by Lemma~\ref{lem:signP}. 

\begin{lemma} \label{lem:signP2}
There is an algorithm (using only rational arithmetic) which takes as input an algebraic number $\gamma$, a real number $\eta$ that is a root of an irreducible polynomial 
over $\mathbb{Q}(\gamma)$, 
and a polynomial $f$ over $\mathbb{Q}(\gamma)$, and determines the sign of the number $f(\eta)$.  
\end{lemma}
\begin{proof}
Suppose that $\eta$ is represented by $(g,I)$. 
It follows that $f(\eta)=0$ if and only if $g$ divides $f$. 
By Sturm's theorem, we can find an interval $[a,b]$ such that $a$ and $b$ are rational, $[a,b] \subset I$ and 
$f$ has no root in $I$. 
Then the sign of  $f(\eta)$ is the sign of $f(a)$.   
\end{proof}

\begin{lemma} \label{lem:rankM2}
There is an algorithm (using only rational arithmetic) which takes as 
input an real algebraic number $\gamma$, a real number $\eta$ that is a root of an irreducible polynomial 
over $\mathbb{Q}(\gamma)$  and a symmetric matrix $\boldsymbol{M}(t,c)$ whose entries are in $\mathbb{Q}[t,c]$, 
and determines the number of the positive eigenvalues and the number of the negative eigenvalues of  $\boldsymbol{M}(\gamma,\eta)$. This 
decides whether 
$\boldsymbol{M}(\gamma,\eta)$  is positive semidefinite. 
\end{lemma}
\begin{proof}
Let $P(t,c,x)$ be the polynomial defined by 
\[
P(t,c,x)=| \boldsymbol{M}(t,c)-x \boldsymbol{I} |. 
\]
Let $P_i(t,c)$ be the coefficient of $x^i$ in $P(x)=P(t,c,x)$.  
By Lemma~\ref{lem:signP2}, we can determine the sign of $P_i(\gamma,\eta)$. 
Using Descartes' rule of signs, 
the number of the positive roots and the number of the negative roots  of $P(x)=P(\gamma,\eta, x)$ 
are determined by the list of the signs of $P_i(\gamma,\eta)$.   
\end{proof}


\begin{lemma} \label{lem:alg_spherical}
There is an algorithm (using only rational arithmetic) which takes as input 
an algebraic number $\gamma$ and a matrix $\boldsymbol{M}(t)$ whose entries are in $\mathbb{Q}[t]$, and decides
whether $\boldsymbol{M}(\gamma)$ is the distance matrix of a spherical set. 
\end{lemma}
\begin{proof}
First we check if $\boldsymbol{M}(\gamma)$ is dissimilarly. 
Let $P(t,a,x)$ be the polynomial defined by 
\[
P(t,a,x)=|-\boldsymbol{M}(t) +a \boldsymbol{J}-x \boldsymbol{I} |
\]
for indeterminates $a$ and $x$.   
Let $P_i(t,a)$ be the coefficient of $x^i$ in $P(x)=P(t,a,x)$. 
Let $Q_i(t)$ be the coefficient of $a^j$ in $P_i(a)=P_i(t,a)$, where $j$ is the largest exponent that 
satisfies the coefficient of $a^j$ is not divisible by the minimal polynomial $f_\gamma$ of $\gamma$.   
If the coefficient of $a^j$ is divisible by $f_\gamma$ for each $j$, 
then we set $Q_i(t)=0$. 
By Lemma~\ref{lem:signP}, we can determine the sign of $Q_i(\gamma)$. 
For sufficient large $a$, we can determine the sign of $P_i(\gamma,a)$:  
$P_i(\gamma,a)=0$ if and only if $Q_i=0$, 
$P_i(\gamma,a)>0$ if and only if $Q_i(\gamma)>0$, and  
$P_i(\gamma,a)<0$ if and only if $Q_i(\gamma)<0$. 
Using Descartes' rule of signs, 
the number $m$ of the negative roots of $P(x)=P(\gamma,a,x)$ for sufficient large $a$
is determined by the list of the signs of $P_i(\gamma,a)$. 
By Lemma~\ref{rem:rank}, 
$m=0$ if and only if  
$\boldsymbol{M}$ is the distance matrix of a spherical set. 
\end{proof}

\begin{lemma} \label{lem:alg_eta}
There is an algorithm (using only rational arithmetic) which takes as input 
an algebraic number $\gamma$ and 
a  Hermitian matrix $\boldsymbol{H} =\boldsymbol{M} +\sqrt{-1} \boldsymbol{A}$, where $\boldsymbol{M}$ and $ \boldsymbol{A}$ are matrices over $\mathbb{Q}(\gamma)$ that satisfy the condition from Theorem~\ref{thm:eta}, and determines a positive real number $\eta=(f,I)$, where $\eta$ is defined in Theorem~\ref{thm:eta} and $f$ is over $\mathbb{Q}(\gamma)$.  
\end{lemma}
\begin{proof}
Let $\boldsymbol{H}(c)$ be the matrix $\boldsymbol{M} + c \sqrt{-1} \boldsymbol{A}$. 
The value $\eta$ is a unique positive number such that $\boldsymbol{P}\boldsymbol{H}(\eta)\boldsymbol{P}$ is positive semidefinite and   ${\rm Rank} (\boldsymbol{P}\boldsymbol{H}(\eta)\boldsymbol{P})<{\rm Rank} (\boldsymbol{P}\boldsymbol{H}(0)\boldsymbol{P})$. 
Let $P(c,x)$ be the polynomial defined by 
\[
P(c,x)=|\boldsymbol{P}\boldsymbol{H}(c)\boldsymbol{P} -x \boldsymbol{I} |
\]
for an indeterminate $x$.   
Let $P_i(c)$ be the coefficient of $x^i$ in $P(x)=P(c,x)$. 
The polynomial $P_i(c)$ is factored into irreducible polynomials over $\mathbb{Q}(\gamma)$ \cite{T76}. 
The rank of $\boldsymbol{P}\boldsymbol{H}(0)\boldsymbol{P}$ is determined by Lemma~\ref{lem:rankM}. 
The value $\eta$ is determined as the smallest positive zero of $\prod_i P_i(c)$ such that the number of sign differences between consecutive nonzero coefficients $P_i(\eta)$ is smaller than that for $P_i(0)$.  
\end{proof}

\begin{lemma} \label{lem:algo_type}
There is an algorithm (using only rational arithmetic) which takes as 
input an simple graph $G$, and determines the type of $G$.   
\end{lemma}
\begin{proof} 
Let $\boldsymbol{A}$ be the adjacency matrix of $G$. 
Let $\lambda_i$ be the $i$-th smallest eigenvalue of $\boldsymbol{A}$, and $m_i$ the multiplicity of $\lambda_i$. 
Indeed there is an algorithm that gives the factorization of an integral polynomial into irreducible polynomials over $\mathbb{Q}$, see \cite{H02}.   
Let $\boldsymbol{M}(t)$ be the matrix defined by $\boldsymbol{M}(t)=-(t+1)\boldsymbol{A} -t \overline{\boldsymbol{A}}$ for an indeterminate $t$. 
By Lemma \ref{lem:rankM}, we can determine  ${\rm Rank}(\boldsymbol{M}(\lambda_i))$ and ${\rm Rank}(\boldsymbol{P}\boldsymbol{M}(\lambda_i)\boldsymbol{P})$. 
By Lemma~\ref{rem:rank2},  Remark~\ref{rem:Rank(H)}, and Theorems~\ref{thm:dim}, \ref{thm:spherical}, we can determine the type of $G$ as follows. 
$G$ is Type~(1) if and only if ${\rm Rank}(\boldsymbol{P}\boldsymbol{M}(\lambda_1)\boldsymbol{P})=n-m_1-1$ and $\boldsymbol{M}(\lambda_1)$ is the distance matrix of a spherical set.  
$G$ is Type~(2) if and only if $m_1>1$, ${\rm Rank}(\boldsymbol{P}\boldsymbol{M}(\lambda_1)\boldsymbol{P})=n-m_1$, and $\boldsymbol{M}(\lambda_1)$ is the distance matrix of a spherical set. 
$G$ is Type~(3) if and only if $m_1=1$, $\lambda_2<-1$, $\boldsymbol{M}(\lambda_2)$ is not the distance matrix of a spherical set, $\boldsymbol{P}\boldsymbol{M}(\lambda_2)\boldsymbol{P}$ is positive semidefinite, and ${\rm Rank}(\boldsymbol{P}\boldsymbol{M}(\lambda_2)\boldsymbol{P})=n-m_2-2$. 
$G$ is Type~(4) if and only if $m_1=1$ $\lambda_2<-1$, $\boldsymbol{M}(\lambda_2)$ is the distance matrix of a spherical set, and ${\rm Rank}(\boldsymbol{P}\boldsymbol{M}(\lambda_2)\boldsymbol{P})=n-m_2-1$.
If $G$ is not of Type (i) for each $i \in \{1,\ldots ,4\}$, then $G$ is Type~(5). 
\end{proof}

\begin{lemma} \label{lem:algo_rep(G)}
Let $G$ be a digraph with adjacency matrix $\boldsymbol{A}$. Let $G'$ be either the simple graph with the 
adjacency matrix $\boldsymbol{B}=\boldsymbol{A}+\boldsymbol{A}^{T}$ or its complement. Suppose $G'$ is of Type~(1), (2), or (4). 
If the null space of the minimal representation $\xi \boldsymbol{B}+\overline{\boldsymbol{B}}$ is contained in that of $\boldsymbol{A}-\boldsymbol{A}^T$, 
then there is an algorithm (using only rational arithmetic) which determines ${\rm Rep}(G)$. 
\end{lemma}
\begin{proof}
By Lemma~\ref{lem:alg_eta}, 
we can determine $\eta$ such that $-\boldsymbol{P}(\xi \boldsymbol{B}+\overline{\boldsymbol{B}})\boldsymbol{P}+\eta \sqrt{-1} \boldsymbol{P}(\boldsymbol{A}-\boldsymbol{A}^T)\boldsymbol{P}$ is a positive semidefinite matrix of rank less than ${\rm Rep}(G')$.  
Note that ${\rm Rep}(G')$ is determined by Lemma~\ref{lem:algo_type}. 
If there exists a positive number $a$ such that $-(\xi \boldsymbol{B}+\overline{\boldsymbol{B}})+\eta \sqrt{-1} (\boldsymbol{A}-\boldsymbol{A}^T)+a\boldsymbol{J}$ is positive semidefinite, then 
${\rm Rep}(G)$ is the rank of $-\boldsymbol{P}(\xi \boldsymbol{B}+\overline{\boldsymbol{B}})\boldsymbol{P}+\eta \sqrt{-1} \boldsymbol{P}(\boldsymbol{A}-\boldsymbol{A}^T)\boldsymbol{P}$, else ${\rm Rep}(G)={\rm Rep}(G')$ by Theorem~\ref{thm:3-code}. The existence of such $a$ can be checked by a similar manner to Lemma~\ref{lem:alg_spherical}.  Here the signs of coefficients are checked by Lemma~\ref{lem:signP2}.
\end{proof}

We describe the {\it algorithm} to classify the largest 3-codes in $\Omega(d)$. 
We first classify simple graphs $G'$ that may give the oriented graphs $G$ whose representations are the largest 3-codes.   
Let $L_0(\gamma)$ be the all $(2d+2)$-vertex simple graphs $G'$ that represent 2-distance sets in $S^{2d-1}$, with distances $1$ and $\gamma$.  
For $G' \in L_0(\gamma)$,  the representation of $G'$ in $S^{2d-1}$ is the minimal representation.   
The graph in $L_0(\gamma)$ is of Type (1), (2), or (4)  by Theorem~\ref{thm:spherical}.  
The distance $\gamma$ may be less than 1, and $\gamma=(\lambda+1)/\lambda$ holds, where $\lambda$ is the smallest or second-smallest eigenvalue of $G$ by Theorem~\ref{thm:dim}.  
First we produce $L_0(\gamma)$ for any possible $\gamma$ by applying Lemma~\ref{lem:algo_type} to all exhaustive simple graphs with $2d+2$ vertices.  
We have the list of exhaustive simple graphs with at most $10$ vertices \cite{MP13}. 

Let $G'$ be a simple graph in $L_0(\gamma)$.   
Let $\boldsymbol{B}$ be the adjacency matrix of $G'$, and  $\overline{\boldsymbol{B}}$  the adjacency matrix of the compliment.  
Let $\boldsymbol{M}(\lambda)$ be the matrix $(\lambda+1) \boldsymbol{B}+ \lambda \overline{\boldsymbol{B}}$, where $\lambda=1/(\gamma-1)$.  
Let $\mathcal{E}_0'$ be the null space of $\boldsymbol{M}(\lambda)$.
Let $K(G')$ be the set of all oriented graphs $G$ such that 
$\mathcal{E}_0' \subseteq \mathcal{E}_0$, $\boldsymbol{A}+\boldsymbol{A}^T=\boldsymbol{B}$ or 
$\overline{\boldsymbol{B}}$, and ${\rm Rep}(G) \leq d$, where $\boldsymbol{A}$ is the adjacency matrix of $G$ and $\mathcal{E}_0$ be the null space of $\boldsymbol{A}-\boldsymbol{A}^T$.  
Here ${\rm Rep}(G)$ is determined by Lemma~\ref{lem:algo_rep(G)}. 
Note that $\mathcal{E}_0' \subseteq \mathcal{E}_0$ if and only if 
the row space of $\boldsymbol{A}-\boldsymbol{A}^T$ is contained in the row space of $\boldsymbol{M}(\lambda)$. 
Moreover when ${\rm Rank}(\boldsymbol{M}(\lambda))=2d$ we need  ${\rm Rank}(\boldsymbol{A}-\boldsymbol{A}^T)=2d$ in order to have ${\rm Rep}(G)=d$ by the proof of Theorem~\ref{thm:half}.   These conditions can reduce a large number of choices of $\boldsymbol{A}$. 
We can make the list of  $\boldsymbol{A}$ and give ${\rm Rep}(G)$ for each $\boldsymbol{A}$. 
 If $K(G')$ is empty, then $G'$ is removed from $L_0(\gamma)$. 
Note that $L_0(\gamma)$ is not empty because the union of $d$ mutually orthogonal equilateral triangles is a $3$-code with $3d$ points.  

Let $L(n,\gamma)$ be the set of all $n$-vertex simple graphs $G'$ of Type (1), (2) or (4) such that $K(G')$ is not empty. Now $L(2d+2,\gamma)=L_0(\gamma)$. 
The list of $L(n+1,\gamma)$ is produced from $L(n,\gamma)$ by the following algorithm based on \cite{L97}. 
Possibilities of augmenting graph $G' \in L(n,\gamma)$ by an $(n+1)$-th vertex are examined. 
There are $2^n$ possibilities of a newly added $(n+1)$-th row of $\boldsymbol{B}$.  
Its entries are in $\{0, 1\}$.  
We may think of these $2^n$ sequences as leaves of a binary tree of depth $n$. 
In depth at least $2d+2$, the search effectively pruned by checking various sub-matrices of size 
$2d+2$ against the list $L(2d+2,\gamma)$.  
Let $\tilde{\boldsymbol{B}}$ be a new matrix obtained from $\boldsymbol{B}$ by adding a new column and a new row, 
and $\tilde{G'}$ the simple graph with the adjacency matrix $\tilde{\boldsymbol{B}}$. 
We check whether $\tilde{G'}$ already appears in $L(n+1, \gamma)$. 
If not, then we form the $2d+2$ graphs $\tilde{G'}_i$ for $1 \leq i \leq 2d+2$, where 
$\tilde{G'}_i$ is the induced subgraph of $\tilde{G'}$ which 
arises by deleting its vertex $i$. 
Since any induced subgraph of $\tilde{G'}$ on $2d+2$ vertices is contained in at least one of the graphs 
$\tilde{G'}_1,\ldots,\tilde{G'}_{2d+2}$, $G'$, it follows that 
${\rm Rep}(\tilde{G'})\leq 2d$ if and only if all graphs $\tilde{G'}_1,\ldots,\tilde{G'}_{2d+2}$, $G'$ are appears in $L(n,\gamma)$. 
If $\tilde{G'}$ is of Type (1), (2), or (4) and $K(\tilde{G'})$ is not empty, then $\tilde{G'}$ is appended to $L(n,\gamma)$. 

The smallest number $n$ such that $L(n+1,\gamma)$ is empty for any $\gamma$ is the size of 
a largest 3-code. 
For all $G'$ in $L(n,\gamma)$, the union of the sets $K(G')$ gives the classification of oriented graphs whose complex representations are largest $3$-codes.  

By the {\it algorithm} we can classify the largest 
complex $3$-codes in
$\Omega(d)$ for $d=1,2,3$. 
Table 1 shows the number of largest $3$-codes. 
\begin{center}
$
\begin{array}{c|ccc}
	d   & 1 & 2 & 3   \\ \hline 
	|X| & 4 & 8 & 9   \\
	\#    & 1 & 1 & 50 
\end{array}
	$\\
Table 1
\end{center}
For $d\geq 4$, a usual computer cannot give the classification. 
For $d=1,2$, the largest complex $3$-codes are tight, and they are considered in 
Section~\ref{sec:tight}. 
For $d=3$, 
one of the largest $3$-codes is the union of 
three equilateral triangles in $\mathbb{C}^1$, which are orthogonal to each other. 
For the other largest $3$-codes $X$,  
$\phi(X \cup e^{2\pi \sqrt{-1}/3} X \cup e^{4\pi \sqrt{-1}/3} X)$ is the unique largest $2$-distance set in $\mathbb{R}^6$ \cite{DGS77,S68}, 
which is the minimal representation of the Schl${\rm \ddot{a}}$fli
graph with $27$ vertices.  

\section{Tight complex spherical $3$-codes} \label{sec:tight}
In this section, we give upper bounds on complex spherical $3$-codes and characterize $3$-codes achieving the upper bound by using another type of codes, called $\mathcal{S}$-codes. 
A tight $\mathcal{S}$-code with degree $|\mathcal{S}|-1$ has the structure of a commutative association scheme. 
We review the theory of complex spherical designs and codes \cite{RSX} and commutative association schemes \cite{BI}.   

Let $\nn$ denote the set of nonnegative integers.  
A finite subset $\mathcal{S}$ of $\mathbb{N}^2$ is a {\it lower set} if the following condition is satisfied: if $(i,j)\in\mathbb{N}^2$ is in $\mathcal{S}$ then so is $(k,l)$ for any $0\leq k\leq i$ and $0\leq l \leq j$. 
A finite set $X$ in $\Omega(d)$ is an {\it $\mathcal{S}$-code} if there exists a polynomial $F(x)=\sum_{(k,l)\in\mathcal{S}}a_{k,l}x^k\bar{x}^l$ with real coefficients such that $F(\alpha)=0$ for any $\alpha\in D(X)$ and $F(1)> 0$.

We denote by $\Hom_d(k,l)$ the vector space generated by homogeneous polynomials of degree $k$ in variables $\{z_1,\ldots,z_d\}$ and of degree $l$ in variables $\{\bar{z}_1,\ldots,\bar{z}_d\}$.
The unitary group $U(d)$ acts on $\Hom_d(k,l)$, 
and the irreducible decomposition is 
\begin{align*}
\Hom_d(k,l)=\bigoplus_{m=0}^{\min(k,l)}\Harm_d(k-m,l-m),   
\end{align*}
where ${\rm Harm}(k,l)$ is the subspace
of ${\rm Hom}(k,l)$ that is the kernel of the Laplace operator
$\Delta=\sum_{i=1}^d\partial^2/\partial z_i\partial\overline{z_i}$.

Define an inner product on polynomials $f$ and $g$ on $\Om(d)$ as follows:
\[
\ip{f}{g} := \int_{\Om(d)} \overline{f(\boldsymbol{z})}g(\boldsymbol{z}) \dd \boldsymbol{z}.
\]
Here $\mathrm{d} \boldsymbol{z}$ is the unique invariant Haar measure on $\Om(d)$, normalized so that $\int_{\Om(d)}\mathrm{d} \boldsymbol{z} = 1$. With respect to this inner product, $\Harm_d(k,l)$ is orthogonal to $\Harm_d(k',l')$ whenever $(k,l) \neq (k',l')$. 
For each $(k,l)\in\mathbb{N}^2$, fix an orthonormal basis $\{e_1,\ldots,e_{m^d_{k,l}}\}$ for the space $\Harm_d(k,l)$.
For a finite set $X$ in $\Omega(d)$, we define the characteristic matrix $\boldsymbol{H}_{k,l}$ with rows indexed by $X$ and columns indexed by $\{1,2,\ldots,m^d_{k,l}\}$ as 
\begin{align*}
(\boldsymbol{H}_{k,l})_{\boldsymbol{x},i}=e_i(\boldsymbol{x})
\end{align*}
for $\boldsymbol{x}\in X$ and $i\in\{1,2,\ldots,m^d_{k,l}\}$.

For each $(k,l) \in \nn^2$, we define a Jacobi polynomial $g_{k,l}^d$ as follows:
\[
g_{k,l}^d(x) := \frac{m_{k,l}^d(d-2)!k!l!}{(d+k-2)!(d+l-2)!} \sum_{r=0}^{\min\{k,l\}} (-1)^r \frac{(d+k+l-r-2)!}{r!(k-r)!(l-r)!}x^{k-r}\overline{x}^{l-r}, 
\]
where 
\begin{align}
m_{k,l}^d&=\dim(\Harm_d(k,l))\nonumber \\
& = \binom{d+k-1}{d-1}\binom{d+l-1}{d-1} - \binom{d+k-2}{d-1}\binom{d+l-2}{d-1}.\label{eq:dim}
\end{align}
The Jacobi polynomials which we used are 
\begin{align*}
g_{0,0}^d(x) & = 1, \\
g_{1,0}^d(x) & = d x,  \displaybreak[0]\\
g_{0,1}^d(x) & = d \overline{x},  \displaybreak[0]\\
g_{1,1}^d(x) & = (d+1)(d x\overline{x} - 1). 
\end{align*}
Recursively, the Jacobi polynomials satisfy
\begin{align}\label{eq:rec}
x g_{k,l}^d(x) = a_{k,l}g_{k+1,l}^d(x)+b_{k,l}g_{k,l-1}^d(x), 
\end{align}
where $a_{k,l}=\frac{k+1}{d+k+l}$, $b_{k,l}=\frac{d+l-2}{d+k+l-2}$ and set $g_{k,l}^d(x)=0$ unless $(k,l) \in \nn^2$.

The essential property of the Jacobi polynomials is the following theorem, known as Koornwinder's addition theorem.

\begin{theorem}\label{thm:addition}
Let $\{e_1,\ldots,e_{m^d_{k,l}}\}$ be an orthonormal basis for the space $\Harm_d(k,l)$. Then for any $\boldsymbol{a},\boldsymbol{b} \in \Om(d)$,
\[
\sum_{i=1}^{m^d_{k,l}} \overline{e_i(\boldsymbol{a})}e_i(\boldsymbol{b}) = g_{k,l}^d(\boldsymbol{a}^*\boldsymbol{b}).
\]
\end{theorem}

An upper bound on the size of an $\mathcal{S}$-code is given as follows.
\begin{theorem}[{\cite[Theorem~4.2 (ii)]{RSX}}]\label{thm:42}
For $d\geq 2$, let $X$ be an $\mathcal{S}$-code in $\Omega(d)$. Then $|X|\leq \sum_{(k,l)\in\mathcal{S}}\dim({\rm Harm}(k,l))$ holds.
\end{theorem}

An $\mathcal{S}$-code is {\it tight} if equality holds in Theorem~\ref{thm:42}. 
Tight codes are related to complex spherical designs. 
For a finite lower set $\mathcal{T}$, a finite subset $X$ of $\Omega(d)$ is a {\it complex spherical $\mathcal{T}$-design} if, for every polynomial $f \in \mathrm{Hom}(k,l)$ such that $(k,l)$ is in $\mathcal{T}$, 
\begin{equation}
\frac{1}{|X|} \sum_{\boldsymbol{z} \in X} f(\boldsymbol{z}) = \int_{\Omega(d)} f(\boldsymbol{z}) \mathrm{d} \boldsymbol{z},
\end{equation}
where $\mathrm{d}\boldsymbol{z}$ is the Haar measure on $\Omega(d)$ 
normalized by $\int_{\Omega(d)} \mathrm{d} \boldsymbol{z}=1$.
As stated in the following theorem, tight $\mathcal{S}$-codes are complex spherical $\mathcal{S}*\mathcal{S}$-designs, 
where $\mathcal{S}*\mathcal{S}:=
\{(k+l',k'+l)\mid (k,l), (k',l') \in \mathcal{S}\}$. 

\begin{theorem}[{\cite[Theorem 5.4]{RSX}}]\label{thm:tight}
Let $X$ be a finite  set in $\Omega(d)$ and let $\mathcal{S}$ be a lower set. Then the following are equivalent:
\begin{enumerate}
\item $X$ is a tight $\mathcal{S}$-code. 
\item $X$ is a tight $\mathcal{S}*\mathcal{S}$-design.
\item $X$ is an $\mathcal{S}$-code and an $\mathcal{S} * \mathcal{S}$-design.
\end{enumerate}
\end{theorem} 
An $\mathcal{S}*\mathcal{S}$-design satisfies that $|X|\geq \sum_{(k,l)\in\mathcal{S}}\dim({\rm Harm}(k,l))$, and an $\mathcal{S}*\mathcal{S}$-design $X$ is {\it tight} if the equality is attained.

Let $X$ have an angle set $D(X)=\{\alpha_1,\ldots,\alpha_s\}$, and set $\alpha_0=1$. 
For $0\leq i\leq s$, define the binary relation $R_i$ as the set of pairs $(\boldsymbol{x},\boldsymbol{y})\in X\times X$ such that $\boldsymbol{x}^*\boldsymbol{y}=\alpha_i$.
The following is a key theorem to characterize tight $3$-codes.  

\begin{theorem}[{\cite[Theorem 6.1]{RSX}}]\label{thm:S1}
Let $X$ be a tight $\mathcal{S}$-design with degree $s=|\mathcal{S}|-1$ for a lower set $\mathcal{S}$. Then $X$ with binary relations defined from angles is a commutative association scheme. 
Moreover, the primitive idempotents are $\frac{1}{|X|}\boldsymbol{H}_{k,l}\boldsymbol{H}_{k,l}^*$, $(k,l)\in \mathcal{S}$.  
\end{theorem}
\begin{remark}
If $X$ is a finite set in $\Omega(d)$, then the Gram matrix $\boldsymbol{G}=(\boldsymbol{x}^*\boldsymbol{y})_{\boldsymbol{x},\boldsymbol{y}\in X}$ is  $\frac{1}{d}\boldsymbol{H}_{0,1}\boldsymbol{H}_{0,1}^*$. 
\end{remark}
To characterize the tight $3$-codes, we use the theory of commutative association schemes.
 
Let $X$ be a finite set and let $R_i$ be a nonempty binary relation on $X$ for $i\in\{0,1,\ldots,s\}$. The {\it adjacency matrix} $\boldsymbol{A}_i$ of relation $R_i$ is defined to be the $(0,1)$-matrix with rows and columns indexed by $X$ such that $(\boldsymbol{A}_i)_{xy}=1$ if $(x,y)\in R_i$ and $(\boldsymbol{A}_i)_{xy}=0$ otherwise. 
A pair $(X,\{R_i\}_{i=0}^s)$ is a {\it commutative association scheme}, or simply an {\it association scheme} if the following five conditions hold:
\begin{enumerate}
\item $\boldsymbol{A}_0$ is the identity matrix.
\item $\sum_{i=0}^s\boldsymbol{A}_i=\boldsymbol{J}$, where $\boldsymbol{J}$ is the all-one matrix.
\item For any $i\in\{0,1,\ldots,s\}$, there exists $i'\in\{0,1,\ldots,s\}$ such that $\boldsymbol{A}_i^T=\boldsymbol{A}_{i'}$.
\item For any $i,j,k\in\{0,1,\ldots,s\}$, there exists $p_{i,j}^k$ such that $\boldsymbol{A}_i\boldsymbol{A}_j=\sum_{k=0}^sp_{i,j}^k\boldsymbol{A}_k$.
\item $\boldsymbol{A}_i\boldsymbol{A}_j=\boldsymbol{A}_j\boldsymbol{A}_i$ for any $i,j$.
\end{enumerate}
The algebra $\mathcal{A}$ generated by all adjacency matrices $\boldsymbol{A}_0,\boldsymbol{A}_1,\ldots,\boldsymbol{A}_s$ over $\mathbb{C}$ is called the 
{\it Bose-Mesner algebra}. 

Since the  Bose-Mesner algebra is semisimple and commutative, 
there exists a unique set of primitive idempotents of the  Bose-Mesner algebra, which is denoted by $\{\boldsymbol{E}_0,\boldsymbol{E}_1,\ldots,\boldsymbol{E}_s\}$ \cite[Theorem 3.1]{BI}. 
Since $\{\boldsymbol{E}_0^T,\boldsymbol{E}_1^T,\ldots,\boldsymbol{E}_s^T\}$ forms also the set of primitive idempotents, 
we define $\hat{i}$ by the index such that $\boldsymbol{E}_{\hat{i}}=\boldsymbol{E}_i^T$ for $0\leq i\leq s$. 
Note that $\hat{0}=0$. 
The Bose-Mesner algebra is closed under the entrywise product $\circ$.  
We define structure constants, the {\it Krein parameters} $q_{i,j}^k$, for $\boldsymbol{E}_0,\boldsymbol{E}_1,\ldots,\boldsymbol{E}_s$ under entrywise product:
$$|X|\boldsymbol{E}_i\circ |X|\boldsymbol{E}_j=|X|\sum_{k=0}^sq_{i,j}^k\boldsymbol{E}_k.$$
By the commutativity of the entrywise product, $q_{i,j}^k=q_{j,i}^k$ holds for any $i,j$. 
We need the following fundamental properties on Krein parameters in the proof of Theorem~\ref{thm:chartight}. 
\begin{lemma}\label{lem:as}
Let $(X,\{R_i\}_{i=0}^s)$ be a commutative association scheme of class $s$.
Let $q_{i,j}^k$ be its Krein parameters.  
Then the following hold for any $i,j,k,l$. 
\begin{enumerate}
\item\label{it:1} $q_{i,j}^k\geq0$. 
\item\label{it:2} $q_{i,0}^k=\delta_{i,k}$. 
\item\label{it:3} $q_{i,j}^0=m_i\delta_{i,\hat{j}}$.
\item\label{it:4} $\sum_{j=0}^s q_{i,j}^k=m_i$. 
\item\label{it:5} $m_k q_{i,j}^k=m_{\hat{j}} q_{i,\hat{k}}^{\hat{j}}$. 
\item\label{it:6} $\sum_{\alpha=0}^s q_{i,j}^{\alpha}q_{k,\alpha}^{l}=\sum_{\beta=0}^s q_{k,i}^{\beta}q_{\beta ,j}^{l}$.
\end{enumerate} 
\end{lemma}
\begin{proof}
See \cite[Proposition 3.7, Theorem 3.8]{BI}.
\end{proof}
The matrix $\boldsymbol{B}_i^*=(q_{i,j}^{k})_{j,k=0}^{s}$ is called the {\it Krein matrix} for $i\in\{0,1,\ldots,s\}$.

Both sets of matrices $\{\boldsymbol{A}_0,\boldsymbol{A}_1,\ldots,\boldsymbol{A}_s\}$ and $\{\boldsymbol{E}_0,\boldsymbol{E}_1,\ldots,\boldsymbol{E}_s\}$ are bases for the Bose-Mesner algebra. Therefore there exist change of basis matrices $\boldsymbol{P}$ and $\boldsymbol{Q}$ defined as follows;
\[
\boldsymbol{A}_i=\sum_{j=0}^s\boldsymbol{P}_{ji}\boldsymbol{E}_j,\quad \boldsymbol{E}_j=\frac{1}{|X|}\sum_{i=0}^s\boldsymbol{Q}_{ij}\boldsymbol{A}_i.
\]
Then we have $\boldsymbol{P}=\frac{1}{|X|}\boldsymbol{Q}^{-1}$. 
We call $\boldsymbol{P}$ and $\boldsymbol{Q}$ the {\it eigenmatrix} and {\it second eigenmatrix} of the association scheme, respectively. 
For each $i\in\{0,1,\ldots,s\}$, $k_i:=\boldsymbol{P}_{i0}$ and $m_i:=\boldsymbol{Q}_{i0}$ are called the $i$-th valency and multiplicity, respectively.

The Krein matrices $\boldsymbol{B}_i^*$ and the second eigenmatrix $\boldsymbol{Q}$ are related as follows. The proof is essentially same as that of \cite[Theorem 4.1]{BI}.  
A vector $\boldsymbol{v}$ is {\it standard} if the first entry of $\boldsymbol{v}$ is $1$. 
\begin{lemma}\label{lem:eigen}
Let $(X,\{R_i\}_{i=0}^s)$ be a commutative association scheme with the Krein matrices $\boldsymbol{B}_i^*$ and the second eigenmatrix $\boldsymbol{Q}$. Let $\boldsymbol{v}_i=(\boldsymbol{Q}_{i0},\boldsymbol{Q}_{i1},\ldots,\boldsymbol{Q}_{is})$ be the $i$-th row of $\boldsymbol{Q}$ for $i\in\{0,1,\ldots,s\}$. 
Then $\boldsymbol{v}_i^T$ is characterized as the unique standardized common right eigenvector $\boldsymbol{v}^T$ of the Krein matrices $\boldsymbol{B}_j^*$ such that $\boldsymbol{B}_j^*\boldsymbol{v}^T=\boldsymbol{Q}_{ij}\boldsymbol{v}^T$.   
\end{lemma}
\begin{proof}
Regard the left multiplication with respect to the entrywise product $\circ$ as linear transformation and express them in matrix form with respect to $\{\boldsymbol{E}_0,\boldsymbol{E}_1,\ldots,\boldsymbol{E}_s\}$. 
Then we have an algebra homomorphism $\varphi$ from the  Bose-Mesner algebra to $\mathrm{Mat}_{s+1}(\mathbb{C})$ defined by $\varphi(\boldsymbol{E}_i)=({\boldsymbol{B}_i^*})^T$. 
The rest of the proof is obtained by replacing the roles $\boldsymbol{A}_i,\boldsymbol{P}$ with $\boldsymbol{E}_i,\boldsymbol{Q}$ respectively in the proof of \cite[Theorem 4.1(ii)]{BI}.
\end{proof}

We mention that a complex spherical $s$-code can be obtained from a commutative association scheme of class $s$ as follows. 
Let $\boldsymbol{E}_i$ be a primitive idempotent of the commutative association scheme such that $\boldsymbol{E}_i^T\neq \boldsymbol{E}_i$ and $\boldsymbol{E}_i$ has no repeated rows. 
Since the primitive idempotent is positive semidefinite Hermitian matrices, there exists a $|X|\times m_i$ matrix $\boldsymbol{F}$ such that $\boldsymbol{F}\boldsymbol{F}^T=
(1/m_i|X|)\boldsymbol{E}_i$. 
Then the set $X$ of the column vectors of $\boldsymbol{F}$ forms a finite set in $\Omega(m_i)$ such that $D(X)=\{ \boldsymbol{Q}_{ji}/\boldsymbol{Q}_{0i} \mid 1\leq j\leq s\}$.  
We give an example of complex $3$-codes in this manner. 
This example is not tight, but has large cardinality.     
\begin{example}\label{eq:LM}
In \cite{LM88}, an infinite family of certain distance-regular digraphs of girth $4$ was constructed. 
Note that a distance-regular digraph of girth $s+1$ corresponds to a commutative association scheme of class $s$ with the adjacency matrices determined from the path length in digraphs \cite{D81}.
The commutative association scheme of class $3$ has the following second eigenmatrix \cite{EM87}: 
\begin{align*}
\boldsymbol{Q}=\begin{pmatrix}
1 & \mu(2\mu^2-1) & (2\mu^2-1)(2\mu^2-2\mu+1) & \mu(2\mu^2-1)\\
1 & \mu^2-\mu+\mu^2\sqrt{-1} & -(2\mu^2-2\mu+1) & \mu^2-\mu-\mu^2\sqrt{-1} \\
1 & -\mu & 2\mu-1 & -\mu \\
1 & \mu^2-\mu-\mu^2\sqrt{-1} & -(2\mu^2-2\mu+1) & \mu^2-\mu+\mu^2\sqrt{-1} 
\end{pmatrix},
\end{align*}
where $\mu$ is any power of $2$. 
Then the primitive idempotent $\boldsymbol{E}_1$ yields a complex spherical $3$-code $X$ in $\Omega(\mu(2\mu^2-1))$ with $|X|=4\mu^4$ and \[
D(X)=\left\{\frac{\mu-1\pm\mu \sqrt{-1}}{2\mu^2-1},\frac{-1}{2\mu^2-1}\right\}.
\] 
\end{example}

\subsection{Tight complex spherical $3$-codes}
\label{subsec:large_3-code}
Let $X$ be a $3$-code in $\Omega(d)$ with $D(X)=\{\alpha,\overline{\alpha},\beta\}$, where $\alpha$ is an imaginary number and $\beta$ is a real number.
Note that $\phi(X)$ is a real $s$-code with $s=1$ or $2$.
When $d=1$, $|X|=|\phi(X)|\leq 5$ with equality if and only if $\phi(X)$ is the regular $5$-gon \cite{DGS77}.
In this case, $X$ has the following angle set $\{e^{2\pi i/5}:0\leq i\leq 4\}$, which implies that $X$ has degree $4$.
Thus $|X|\leq 4$ holds.
When $d\geq2$, we can easily find real numbers $a,b,c$ such that $F(x)=a x \overline{x} +b(x+\overline{x})+c$ is an annihilator polynomial of $X$. 
This implies that $X$ is an $\mathcal{S}$-code, where $\mathcal{S}=\{(0,0),(1,0),(0,1),(1,1)\}$. 
By Theorem~\ref{thm:42} with equation \eqref{eq:dim}, we have the following upper bound for $3$-codes. 
\begin{theorem}\label{thm:3codebound}
Let $X$ be a $3$-code in $\Omega(d)$.
Then $$|X|\leq
\begin{cases}
4 &\text{if } d=1, \\
d^2+2d &\text{if }d\geq2.
\end{cases} $$
\end{theorem}
Note that the example for $d=1$ coincides with the case of $\mu=1$ in Example~\ref{eq:LM}. 
However, a tight $3$-code is rare, shown in the following theorem.
\begin{theorem}\label{thm:chartight}
Let $X$ be a $3$-code in $\Omega(d)$ attaining equality in Theorem~$\ref{thm:3codebound}$.
Then one of the following holds;
\begin{enumerate}
\item $d=1$ and $D(X)=\{\pm\sqrt{-1},-1\}$,
\item $d=2$ and $D(X)=\{\pm\sqrt{-1}/\sqrt{3},-1\}$.
\end{enumerate}
\end{theorem}
\begin{proof}
Let $X$ be a tight $3$-code in $\Omega(1)$ with $D(X)=\{\alpha,\overline{\alpha},\beta\}$. After the unitary operation, we may assume that $1\in X$. Then $X=\{1,\alpha,\overline{\alpha},\beta\}$. Since $\beta$ is a real number, $\beta=-1$. Then $\alpha=\sqrt{-1}$ as desired.  

Let $d$ be an integer at least $2$.
Since $X$ is a tight $\mathcal{S}$-code, $X$ is an $\mathcal{S}*\mathcal{S}$-design by Theorem~\ref{thm:tight}.
Since the degree of $X$ is $3$, $X$ with the binary relations obtained from the angles of $X$ carries a commutative association scheme by Theorem~\ref{thm:S1}.
Then the Gram matrix of $X$ is a scalar multiple of some primitive idempotent of the association scheme, say $\boldsymbol{E}_1$.
And we arrange the ordering of the primitive idempotents so that $\boldsymbol{E}_2=\boldsymbol{E}_1^T$ holds and $\boldsymbol{E}_3$ is a real matrix. 
Then $\hat{1}=2,\hat{2}=1,\hat{3}=3$ hold.

We will determine the Krein matrix $\boldsymbol{B}_1^*$ and the second eigenmatrix $\boldsymbol{Q}$.
We use Lemma~\ref{lem:as} (\ref{it:2}),(\ref{it:3}) 
to obtain $q_{1,0}^0=q_{1,0}^2=q_{1,0}^3=q_{1,1}^0=q_{1,3}^0=0$, $q_{1,0}^1=1$, and $q_{1,2}^0=d$.
By Theorem~\ref{thm:S1},  we may set 
\begin{align*}
\boldsymbol{E}_1&=\frac{1}{|X|}\boldsymbol{H}_{1,0}\boldsymbol{H}_{1,0}^*,\\ 
\boldsymbol{E}_2&=\frac{1}{|X|}\boldsymbol{H}_{0,1}\boldsymbol{H}_{0,1}^*,\\ 
\boldsymbol{E}_3&=\frac{1}{|X|}\boldsymbol{H}_{1,1}\boldsymbol{H}_{1,1}^*. 
\end{align*}

By the recurrence \eqref{eq:rec}, we have that $\boldsymbol{E}_2=\frac{1}{|X|}g_{0,1}\circ(\frac{|X|}{d}\boldsymbol{E}_1)$ and $\boldsymbol{E}_3=\frac{1}{|X|}g_{1,1}\circ(\frac{|X|}{d}\boldsymbol{E}_1)$, 
where $f\circ(\boldsymbol{M})$ denotes the matrix obtained by applying a function $f$ to each entry of a matrix $\boldsymbol{M}$.
By the recurrence \eqref{eq:rec} of the Jacobi polynomial, the Krein parameters $q_{1,2}^1,q_{1,2}^2,q_{1,2}^3$ are the same as the coefficients of the Jacobi polynomials in the product $g_{1,0}(x)g_{0,1}(x)$, namely $q_{1,2}^1=q_{1,2}^2=0$ and $q_{1,2}^3=\frac{d}{d+1}$ holds.
Since $X$ is an $\mathcal{S}*\mathcal{S}$-design and $\mathcal{S}*\mathcal{S}$ contains $(2,1)$,  $q_{1,1}^1=0$ holds by \cite[Corollary~9.3 (ii)]{RSX}.
By Lemma~\ref{lem:as} (\ref{it:4}), we have 
\begin{align}
q_{1,1}^2+q_{1,3}^2&=d,\label{eq:krein1}\\
q_{1,1}^3+q_{1,3}^3&=\tfrac{d^2}{d+1}\label{eq:krein2}.
\end{align}

We have $m_1=\dim(\text{Harm}(1,0))=d$ and $m_3=\dim(\text{Harm}(1,1))=d^2-1$ by \eqref{eq:dim}. 
Substituting the values $m_1$, $m_3$ into the equation in Lemma~\ref{lem:as} (\ref{it:5}) 
for $(i,j,k)=(1,1,3)$,  we have 
\begin{align}
(d^2-1)q_{1,1}^3&=d q_{1,3}^2\label{eq:krein4}.
\end{align}
Using the equation in Lemma~\ref{lem:as} (\ref{it:6}) 
for $(i,j,k,l)=(1,1,2,1)$, we have
\begin{align}
(q_{1,1}^2)^2+\tfrac{d^2-1}{d}q_{1,1}^3q_{1,3}^{2}&=\tfrac{2d^2}{d+1}.\label{eq:krein3}
\end{align}
We solve the equations \eqref{eq:krein1}--\eqref{eq:krein3} to obtain 
\begin{align}\label{eq:55}
 &(q_{1,1}^2,q_{1,1}^3,q_{1,3}^2,q_{1,3}^3)= \notag\\
&\begin{cases} 
(\tfrac{d(d-(d-1)\sqrt{d+2})}{d^2+d-1},\tfrac{d^2(d+1+\sqrt{d+2})}{(d+1)(d^2+d-1)},\tfrac{d(d-1)(d+1+\sqrt{d+2})}{d^2+d-1},\tfrac{d^2(d^2-2-\sqrt{d+2})}{(d+1)(d^2+d-1)}),\\
(\tfrac{d(d+(d-1)\sqrt{d+2})}{d^2+d-1},\tfrac{d^2(d+1-\sqrt{d+2})}{(d+1)(d^2+d-1)},\tfrac{d(d-1)(d+1-\sqrt{d+2})}{d^2+d-1},\tfrac{d^2(d^2-2+\sqrt{d+2})}{(d+1)(d^2+d-1)}).
\end{cases}
\end{align}

First we consider the former case in \eqref{eq:55}.
Since the Krein number $q_{1,1}^2$ is nonnegative by Lemma~\ref{lem:as} (\ref{it:1}), we must have  $d=2$.
In this case the second eigenmatrix $\boldsymbol{Q}$ is given by Lemma~\ref{lem:eigen} as 
\begin{align*}
\boldsymbol{Q}=\begin{pmatrix}
1&2&2&3\\
1&\tfrac{2\sqrt{-1}}{\sqrt{3}}&-\tfrac{2\sqrt{-1}}{\sqrt{3}}&-1 \\
1&-\tfrac{2\sqrt{-1}}{\sqrt{3}}&\tfrac{2\sqrt{-1}}{\sqrt{3}}&-1 \\
1&-2&-2&3
\end{pmatrix}.
\end{align*}
Thus we have that $X$ is a complex $3$-code with $D(X)=\{\pm\sqrt{-1}/\sqrt{3},-1\}$. 

Next, in the latter case in \eqref{eq:55}, we set $t=\sqrt{d+2}$. 
The second eigenmatrix is given by Lemma~\ref{lem:eigen} as
  \begin{align*}
\boldsymbol{Q}=\begin{pmatrix}
1& t^2-2 & t^2-2 & (t^2-3)(t^2-1) \\
1&\tfrac{t^2-2}{t+1}&\tfrac{t^2-2}{t+1}& 1-2t+\tfrac{2}{t+1} \\
1& \tfrac{(t^2-2)(t^2+t-1+t\sqrt{-3t^2-2t+5})}{2(t^3-2t+1)} & \tfrac{-6 - 3 t + 3 t^2 + 2 t^3-t\sqrt{-3t^2-2t+5}}{4(t^2-1)(t^2+t-1)} & \tfrac{(t+1)(t^2-3)}{t^2+t-1} \\
1& \tfrac{-6 - 3 t + 3 t^2 + 2 t^3-t\sqrt{-3t^2-2t+5}}{4(t^2-1)(t^2+t-1)} & \tfrac{(t^2-2)(t^2+t-1+t\sqrt{-3t^2-2t+5})}{2(t^3-2t+1)} & \tfrac{(t+1)(t^2-3)}{t^2+t-1} 
\end{pmatrix}.
\end{align*}
Then the valency corresponding to the second row of the second eigenmatrix is determined as $k_1=\frac{(t+1)^3(t^2-3)}{3t+5}$ by $\boldsymbol{P}=\frac{1}{|X|}\boldsymbol{Q}^{-1}$.
By substituting $t=\sqrt{d+2}$, we find that the valency $k_1$ is equal to $\frac{(d-1)(3d^2+6d-5+4(d-1)\sqrt{d+2})}{9d-7}$, which implies that $t=\sqrt{d+2}$ must be an integer.
The partial fraction decomposition $243 k_1=81t^4+108t^3-180t^2-348t-149+\frac{16}{3t+5}$ shows that $3t+5$ divides $16$.
Since $t$ is positive, we have $t=1$ and thus $d=-1$.
This contradicts to the fact that $d$ is positive.
\end{proof}
For $d=1,2$, the tight 3-code is unique, that is proved in Section~\ref{sec:algo}. 
The tight $3$-code in $\Omega(1)$ is $X=\{\pm 1,\pm \sqrt{-1}\}$.  
The tight $3$-code in $\Omega(2)$ is $\{\pm x_1,\pm x_2,\pm x_3,\pm x_4\}$, where $x_1=(1,0)$, $x_2=1/\sqrt{6}(\sqrt{-2},1+\sqrt{-3})$, $x_3=1/\sqrt{6}(\sqrt{-2},1-\sqrt{-3})$, $x_4=1/\sqrt{6}(\sqrt{-2},-2)$.
%

\begin{remark}
For $\mathcal{S}=\{(0,0),(1,0),(0,1),(1,1)\}$, 
the tight $\mathcal{S}$-codes with degree $4$ were given in \cite[Example 10.2]{RSX}.
They are obtained from the subconstituents of SIC-POVMs in dimension $d=2,8$. 
SIC-POVMs are the tight projective $1$-codes, see \cite{RBSC04} more details.
\end{remark}

\end{document}